\documentclass[reqno, oneside]{amsart}
\pdfoutput=1
\usepackage[foot]{amsaddr}

\usepackage[utf8]{inputenc}

\usepackage{amsthm, amsmath, amssymb}
\usepackage{mathtools}
\usepackage{thmtools}
\usepackage{thm-restate}
\usepackage[table]{xcolor}
\usepackage{enumitem}
\usepackage[bookmarksnumbered,linktocpage,hypertexnames=false,colorlinks=true,linkcolor=blue,urlcolor=blue,citecolor=blue,anchorcolor=green,breaklinks=true,pdfusetitle]{hyperref}

\usepackage{cleveref}
\crefname{equation}{}{}

\usepackage{tablefootnote}
\usepackage{tikz}

\usepackage{graphicx}
\usepackage{physics}

\usepackage{stmaryrd}

\usepackage{caption}

\newtheorem{theorem}{Theorem}
\newtheorem{corollary}[theorem]{Corollary}
\newtheorem{lemma}[theorem]{Lemma}
\newtheorem{definition}[theorem]{Definition}

\newtheorem{conjecture}[theorem]{Conjecture}
\newtheorem{proposition}[theorem]{Proposition}
\newtheorem{fact}{Fact}

\theoremstyle{definition}

\newcommand{\R}{\mathbb{R}}
\newcommand{\N}{\mathbb{N}}
\newcommand{\Z}{\mathbb{Z}}

\newcommand{\funcmin}{f_{\min}}
\newcommand{\funcmax}{f_{\max}}

\newcommand{\mg}{\mathbf{g}}

\newcommand{\cube}[1]{\mathrm{B}^{#1}}
\newcommand{\D}{\mathrm{D}}
\newcommand{\quadr}{\mathcal{Q}}
\newcommand{\preo}{\mathcal{T}}
\newcommand{\nnon}{\mathcal{P}_{\geq 0}}
\newcommand{\pos}{\mathcal{P}_{> 0}}

\newcommand{\x}{\mathbf{x}}
\newcommand{\y}{\mathbf{y}}
\newcommand{\z}{\mathbf{z}}

\renewcommand{\epsilon}{\varepsilon}
\newcommand{\cS}{\mathcal{S}}

\newcommand{\metric}[2]{\| #1 \|^{#2}_{#2}}
\newcommand{\Chebyconst}{\mathfrak{c}}

\newcommand{\Cheby}{T}
\title[Putinar's Positivstellensatz on the hypercube]%
{Degree bounds for Putinar's Positivstellensatz \\ on the hypercube}

\author{Lorenzo Baldi\textsuperscript{*} \and Lucas Slot\textsuperscript{$\dagger$}}
\address{\textsuperscript{*}Sorbonne Universit\'e, CNRS, F-75005, Paris, France}
\email{Lorenzo.Baldi@Lip6.fr}
\address{\textsuperscript{$\dagger$}ETH Z\"urich, Z\"urich, Switzerland}
\email{lucas.slot@inf.ethz.ch}

\date{\today}

\begin{document}

\begin{abstract}
The Positivstellens\"atze of Putinar and Schm\"udgen show that any polynomial $f$ positive on a compact semialgebraic set can be represented using sums of squares. Recently, there has been large interest in proving effective versions of these results, namely to show bounds on the required degree of the sums of squares in such representations. These \emph{effective} Positivstellens\"atze have direct implications for the convergence rate of the celebrated moment-SOS hierarchy in polynomial optimization. In this paper, we restrict to the fundamental case of the hypercube $\cube{n} = [-1, 1]^n$. We show an upper degree bound for Putinar-type representations on~$\cube{n}$ of the order $O(\funcmax/\funcmin)$, where $\funcmax$, $\funcmin$ are the maximum and minimum of $f$ on~$\cube{n}$, respectively. Previously, specialized results of this kind were available only for Schm\"udgen-type representations and not for Putinar-type ones.
Complementing this upper degree bound, we show a lower degree bound in $\Omega(\sqrt[8]{\funcmax/\funcmin})$. This is the first lower bound for Putinar-type representations on a semialgebraic set with nonempty interior described by a standard set of inequalities.
\end{abstract}

\maketitle

\vspace{-1cm}
\tableofcontents

\section{Introduction}

Let $S(\mathbf{g}) \subseteq \R^n$ be a (basic, closed) semialgebraic set, defined in terms of the tuple of polynomials $\mathbf{g} = (g_1, g_2, \ldots, g_m)$ as:
\[
    \cS(\mathbf{g}) = \{ \x \in \R^n : g_1(\x) \geq 0, \ldots, g_m(\x) \geq 0\}.
\]
Consider the problem of determining whether a given polynomial $f$ belongs to the cone $\nnon(\cS(\mg))$ of polynomials nonnegative on $\cS(\mg)$. In general, this is a hard problem. In the unconstrained case, a straightforward way of certifying nonnegativity of $f$ on $\R^n$ is to write
\[
    f(\x) = p_1(\x)^2 + p_2(\x)^2 + \ldots + p_\ell(\x)^2,
\]
i.e., to write $f$ as a \emph{sum of squares} of polynomials. Indeed, the cone $\Sigma[\x]$ of such polynomials is clearly contained in $\nnon(\R^n)$. This idea extends to the constrained case by considering the \emph{quadratic module}~$\quadr(\mathbf{g})$ and \emph{preordering}~$\preo(\mathbf{g})$ of $\mg$, given respectively by:
\begin{align}
    \label{EQ:qmod}
    \quadr(\mathbf{g}) &= \big\{ \, \sum_{i=0}^m \sigma_{i} g_i : \sigma_i \in \Sigma[\x], \quad i = 0, 1, \ldots, m \, \big\}, \\
    \label{EQ:preo}
    \preo(\mathbf{g}) &= \big\{ \, \sum_{I \subseteq [m]} \sigma_{I} g_I : \sigma_I \in \Sigma[\x], \quad I \subseteq [m] \, \big\}.
\end{align}
Here, $g_I := \prod_{i \in I} g_i$ for $I \subseteq [m] = \{1, 2, \ldots, m\}$, and we have adopted the convention that $g_0 = g_\emptyset = 1$. Note that the quadratic module generated by $\mg$ is contained in the preordering, and that they are both contained in $\nnon(\cS(\mg))$.
General representations for nonnegative polynomials on semialgebraic sets have been provided by Krivine \cite{krivine_anneaux_1964} and Stengle
\cite{stengle_nullstellensatz_1974}: these representations use \emph{ratios} of polynomials in the preordering, and extend Artin's solution to Hilbert's seventeenth problem \cite{artin_uber_1927}.

A natural question is then whether all nonnegative polynomials on $\cS(\mg)$ admit a denominator-free representation, i.e., whether they lie in $\preo(\mg)$ or even in $\quadr(\mg)$. While this is not true in general, the Positivstellens\"atze of Putinar (under mild conditions on~$\mg$) and Schm\"udgen show that this is the case if one restricts to the cone $\pos(\cS(\mg))$ of \emph{strictly positive} polynomials on a compact semialgebraic set $\cS(\mg)$.
\begin{theorem}[Schm\"udgen's Positivstellensatz~\cite{Schmudgen:positiv}]
\label{THM:Schmudgen}
Assume that $\cS(\mg) \subseteq \R^n$ is compact. We then have:
\[
    \pos(\cS(\mg)) \subseteq \preo(\mathbf{g}).
\]
\end{theorem}
\begin{theorem}[Putinar's Positivstellensatz~\cite{Putinar:positiv}]
\label{THM:Putinar}
Assume that $\quadr(\mg)$ is Archimedean, i.e, that $R - x_1^2 - \ldots - x_n^2 \in \quadr(\mg)$ for some $R \geq 0$. We then have:
\[
    \pos(\cS(\mg)) \subseteq \quadr(\mathbf{g}).
\]
\end{theorem}

Clearly, semialgebraic sets associated with Archimedean quadratic modules are compact, but this condition is not equivalent to compactness: there exist non-Archimedean quadratic modules that define compact semialgebraic sets, see e.g. \cite[Ex. 6.3.1]{positivbook}. On the other hand, \Cref{THM:Schmudgen} shows that a preordering $\preo (\vb g)$ is Archimedean if and only if the semialgebraic set $\cS(\vb g)$ is compact.

Recently, there has been a substantial interest in proving \emph{effective} versions of the theorems above. This means to show bounds on the minimum degree $r$ so that a positive polynomial $f$ lies in the \emph{truncated} quadratic module or preordering, that are defined, using the convention that $g_0 = g_\emptyset = 1$, as: 
\begin{align}
    \label{EQ:qmodtrunc}
    \quadr(\mathbf{g})_r &= \big\{ \, \sum_{i=0}^m \sigma_{i} g_i : \sigma_i \in \Sigma[\x],~\mathrm{deg}(\sigma_ig_i) \leq r, \quad i =0,1, \dots, m \, \big\}, \\
    \label{EQ:preotrunc}
    \preo(\mathbf{g})_r &= \big\{ \, \sum_{I \subseteq [m]} \sigma_{I} g_I : \sigma_I \in \Sigma[\x],~\mathrm{deg}(\sigma_Ig_I) \leq r, \quad I \subseteq [m] \, \big\}.
\end{align}
Such bounds have immediate implications for the convergence rate of the celebrated moment-SOS hierarchy~\cite{Lasserre:lowbound, Parillo:thesis} for polynomial optimization (see also \Cref{SEC:polyopt}). The Putinar-type representations are of particular interest, as their corresponding hierarchy leads to bounds which may be computed by solving a semidefinite program of polynomial size in the number of variables $n$ and the number of contraints~$m$. 

\subsection{Our contributions} In this paper, we consider the fundamental special case of the hypercube $[-1, 1]^n$, which can be defined as a semialgebraic set by the inequalities $g_i(\x) = 1-x_i^2 \geq 0$, $i=1,2,\ldots,n$. The associated quadratic module $\quadr(\mg) = \quadr (1-x_1^2, \dots , 1-x_n^2)$ is Archimedean,
as it contains $n - x_1^2 - \dots - x_n^2$, 
and Putinar's Positivstellensatz thus applies in this setting. In this paper we prove an upper bound and a lower bound on the degree required for representations of positive polynomials on the hypercube as elements of the quadratic module~$\quadr(\mg)$.

\begin{theorem}[Upper degree bound]\label{THM:main}
Let $f \in \pos([-1, 1]^n)$ be a polynomial of degree~$d$. Denote by $\funcmax$, $\funcmin$ the maximum and the minimum of $f$ on $[-1, 1]^n$, respectively. Then there exists an absolute constant $\Chebyconst > 0$ such that:
\[
    f \in \quadr(1-x_1^2, \ldots, 1-x_n^2)_{rn} \quad \text{whenever} \quad r \geq     
    4 \Chebyconst \cdot d^2(\log n) \cdot \frac{\funcmax}{\funcmin} + O\left(\frac{\funcmax}{\funcmin}\right)^{1/2}.
\]
\end{theorem}
We give a precise expression for the term $O(\sqrt{\funcmax/\funcmin})$ in \Cref{THM:main2}. See also \Cref{SEC:Discussion} for a related discussion. 
\begin{theorem}[Lower degree bound] \label{THM:negative}
    Let $n \geq 2$. For any $\epsilon > 0$ and $r \in \N$, we have:
    \[
    (1-x_1^2)(1-x_2^2) + \epsilon \in \quadr(1-x_1^2, \ldots, 1-x_n^2)_r \implies r = \Omega(1/\sqrt[8]{\epsilon}).
    \]
\end{theorem}
Note that the function $f(\x) = (1-x_1^2)(1-x_2^2) + \epsilon$ satisfies
$\funcmin = \epsilon$ and ${\funcmax = 1+\epsilon}$.
We could therefore replace $\Omega(1/\sqrt[8]{\epsilon})$ in \Cref{THM:negative} by $\Omega( \sqrt[8]{\funcmax / \funcmin})$.

The same asymptotic results of \Cref{THM:main} and \Cref{THM:negative} hold if we use $1 \pm x_i$, $i = 1, \dots , n$ (another set of standard inequalities defining $[-1, 1]^n$) instead of $1-x_i^2$, see \Cref{SEC:Discussion}.

\medskip\noindent\textbf{Outline.}
The paper is structured as follows. In \Cref{SEC:literature}, we discuss the existing literature on effective Archimedean Positivstellens\"atze and their applications to polynomial optimization. We give detailed versions of our main results and explain their relations to prior works.
In \Cref{SEC:prelim}, we cover some preliminaries, particularly on approximation theory. In \Cref{SEC:positive}, we prove our upper degree bound, \Cref{THM:main}. In \Cref{SEC:negative}, we prove the lower degree bound, \Cref{THM:negative}. We conclude in \Cref{SEC:Discussion} by discussing possible future research directions. \Cref{app::degree_shift} is dedicated to the presentation of explicit polynomial identities exploited in \Cref{SEC:positive}.

\section{Related works and applications}
\label{SEC:literature}
In this section, we explain the relation of our main results to the existing literature and their applications. In particular, we focus on existing effective Archimedean Positivstellens\"atze, for general $\vb g$ and specific for the hypercube. Degree bounds for these theorems are usually stated in terms of a parameter of the form $\|f\| / f_{\min, \cS(\mg)}$, whose inverse intuitively measures how close $f$ is to having a zero on $\cS(\mg)$. Here, $f_{\min, \cS(\mg)} = \min_{\x \in \cS(\mg)} f(\x)$, and $\|\cdot\|$ is a norm on $\R[\x]_{\le d}$. Common choices include the supremum norm on $\cS(\mg)$ (or on a compact domain containing $\cS(\mg)$), denoted $f_{\max, \cS(\vb g)}$, and the coefficient norm $\| \cdot \|_{\rm coef}$, defined in terms of the monomial expansion $f(\x) = \sum_{\alpha} f_\alpha \x^\alpha$ as $ \|f\|_{\rm coef} = \max_{\alpha} |f_\alpha| \cdot \frac{\prod_{i}(\alpha_i!)}{(\sum_i \alpha_i)!}$. For fixed number of variables $n$ and degree $d$ of~$f$, these choices are equivalent.

\subsection{General effective Positivstellens\"atze}

For general constraints $\mg$ that define a compact semialgebraic set $\cS(\mg)$, Schweighofer~\cite{Schweighofer:schmudgen} showed in a seminal work that any positive polynomial on $\cS(\mg)$ has a representation in the preordering $\preo(\mg)_r$ truncated at degree 
\[
r \ge O\left(\frac{\norm{f}_{\rm coef}}{f_{\min, \cS(\mg)}}\right)^{c},
\]
where $c > 0$ is a (possibly large) constant depending on $\mg$.
{
Here and in the following, the $O(\, \cdot \, )$ notation should be understood in the following way: there exists a constant $a = a(n,d, \vb g)$, depending on the number of variables $n$, the degree $d$ and on the inequalities $\vb g$, such that the \emph{smallest} $r$ with the property that $f \in \quadr(\vb g)_r$ for all $n$-variate polynomials $f$ of degree $d$, positive on $\cS(\vb g)$ and such that $\frac{\norm{f}_{\rm coef}}{f_{\min, \cS(\vb g)}}$ is big enough, grows at most as $a \cdot \left(\frac{\norm{f}_{\rm coef}}{f_{\min, \cS(\mg)}}\right)^c$. This implies also that $f \in \quadr(\vb g)_r$ for all $r \ge a \cdot \left(\frac{\norm{f}_{\rm coef}}{f_{\min, \cS(\mg)}}\right)^c$
.}

For the quadratic module, Nie \& Schweighofer~\cite{NieSchweighofer:putinar} showed a degree bound for Archimedean $\quadr (\mg)$ with exponential dependence on $\|f\|_{\rm coef}/\funcmin$.
This result was only recently improved in~\cite{BaldiMourrain:putinar,BaldiMourrainParusinski:putinar} to match Schweighofer's polynomial bound for the preordering (although the exponent $c$ may differ).
\begin{theorem}[{\cite[Cor.~3.3]{BaldiMourrainParusinski:putinar}}]\label{THM:GeneralPutinar}
    Let $\quadr (\mg)$ be an Archimedean quadratic module, and let $f$ be a polynomial of degree $d$ positive on $\cS(\mg)$. Then we have, for fixed $n$ and~$d$,
    \[
        f \in \quadr(\mg)_r \quad \text{ for } \quad r \geq O\bigg (\frac{f_{\max,\D}}{f_{\min, \cS(\vb g)}}\bigg)^{7 \text{Ł} + 3}
    \]
    where $\D$ is a scaled simplex containing $\cS(\vb g)$ and $\text{Ł}= \text{Ł}(\vb g)$ is a constant (called Łojasiewicz exponent) depending only on $\vb g$.
\end{theorem}

The Łojasiewicz exponent can be large even when the number of variables $n$ and the degrees $\deg g_1, \dots, \deg g_m$ of the constraints are fixed, see~\cite{BaldiMourrain:putinar,BaldiMourrainParusinski:putinar}.
However, in regular cases, namely when the constraints $\mg$ satisfy the Constraint Qualification Conditions (CQC), one has $\L = 1$.
\begin{definition}[CQC]\label{def::CQC}
We say that a tuple of polynomials $\vb g$ satisfies the constraint qualification conditions if, for every $\x \in \cS(\vb g)$, the gradients of the active constraints at $\x$:
\[
    {\{\grad g(\x) : g \in \mg,~g(\x) = 0\}}
\]
are linearly independent (in particular, nonzero).
\end{definition}

\begin{corollary}[{\cite[Thm.~2.11]{BaldiMourrain:putinar}, \cite[Thm.~2.10 and Cor.~3.4]{BaldiMourrainParusinski:putinar}}]
\label{COR:GeneralPutinar}
        Let $\quadr (\mg)$ be an Archimedean quadratic module, and let $f$ be a polynomial of degree $d$ positive on~$\cS(\mg)$. Assume that $\vb g$ satisfies the CQC. Then we can take $Ł = 1$ in \Cref{THM:GeneralPutinar}, and thus, for fixed $n$ and $d$,
    \[
        f \in \quadr(\mg)_r \quad \text{ for } \quad r \geq O\bigg (\frac{f_{\max,\D}}{f_{\min, \cS(\vb g)}}\bigg)^{10}
    \]
    where $\D$ is a scaled simplex containing $\cS(\vb g)$.
\end{corollary}

\subsection{Specialized effective Positivstellens\"atze}
If we restrict to certain fundamental special cases, stronger bounds are known. When $\cS(\mg)$ is the
hypersphere~\cite{FangFawzi:sphere}, the hypercube~\cite{LaurentSlot:hypercube}, the unit ball~\cite{Slot:CD},
or the standard simplex~\cite{Slot:CD}, we have representations of degree $r = O(\sqrt{f_{\max,\cS(\vb g)} / f_{\min, \cS(\vb g)}})$ in the preordering.
For the hypersphere and unit ball, this bound carries over to the quadratic module
(which, in those cases, is equal to the preordering). However, despite the research effort (see, e.g. \Cref{THM:magron} below), no specialized bounds on the minimum degree required for a representation in the quadratic module are known for the hypercube and the standard simplex. In this paper we start filling this gap, providing the first dedicated analysis for the quadratic module of the hypercube.

\subsection{Effective Positivstellens\"atze for the hypercube}
The unit hypercube $\cube{n} := [-1, 1]^n$ is a compact semialgebraic set that is naturally defined as:
\begin{equation} \label{EQ:cuberepr}
    \cube{n} = \{ \x \in \R^n : g_i(\x) \geq 0,~ 1 \leq i \leq n\}, \quad g_i(\x) := 1-x_i^2.
\end{equation}
Throughout the article, we abuse notation and refer to the quadratic module and preordering generated by $1-x_1^2, \dots ,1-x_n^2$ as:
\begin{align*}
\quadr(\cube{n}) &\coloneqq \quadr(1-x_1^2, \dots 1-x_n^2), \\
\preo(\cube{n}) &\coloneqq \preo(1-x_1^2, \dots 1-x_n^2),
\end{align*}
and we denote their truncations  (see \cref{EQ:qmodtrunc} and \cref{EQ:preotrunc}) as $\quadr(\cube{n})_r$ and $\preo(\cube{n})_r$, respectively.

Despite its simplicity, the best available effective version of Putinar's Positivstellensatz for $\cube{n}$ is the general result of~\cite{BaldiMourrainParusinski:putinar}. Indeed, since the constraints $\mg$ in~\cref{EQ:cuberepr} satisfy the CQC, \Cref{COR:GeneralPutinar} gives a degree bound of the order $O((\funcmax/\funcmin)^{10})$.
On the contrary, for Schm\"udgen Positivstellensatz, specialized results are available, and a much stronger bound of the order $O(\sqrt{\funcmax/\funcmin})$ is known. 
\begin{theorem}[{\cite[Cor.~3]{LaurentSlot:hypercube}}]
\label{THM:boxSchmudgen}
Let $f \in \pos(\cube{n})$ be a polynomial of degree $d$, and let $f_{\min}, f_{\max} > 0$ be the minimum and maximum of $f$ on $\cube{n}$, respectively. Then:
\[
    f \in \preo(\cube{n})_{(r+1)n}, \quad \text{for } \quad r \geq \max \bigg\{
     \left( {C(n, d)} \cdot {\frac{\funcmax}{\funcmin}} \right)^{1/2}, ~ \pi d \sqrt{2n} \bigg\}.
\]
Here, $C(n,d)$ is a constant depending polynomially on $n$ (for fixed $d$), and polynomially on $d$ (for fixed $n$). 
\end{theorem}
For ease of exposition, we stated the bound in \Cref{THM:boxSchmudgen} in a (slightly) weaker form than the one of~{\cite[Cor.~3]{LaurentSlot:hypercube}}.
\Cref{THM:boxSchmudgen} improves upon an earlier analysis due to de Klerk \& Laurent~\cite{deKlerkLaurent:hypercube2010}, who established a bound in $O(\funcmax / \funcmin)$.

In the same work\footnote{In fact, they consider there the cube $[0, 1]^n$ defined by the constraints $x_i \geq 0$, $1-x_i \geq 0$, $i \in [n]$, but all statements carry over after a change of variables. See also \Cref{SEC:Discussion}.}, the authors propose the following conjecture (which remains open):
\begin{conjecture}[de Klerk \& Laurent, 2010] 
\label{CONJ:dKL}
For $n \in \N$ even, we have:
\[
    (1-x_1^2)(1-x_2^2) \ldots (1-x_n^2) + \frac{1}{n(n+2)} \in \quadr(\cube{n})_n.
\]
\end{conjecture}
Assuming \Cref{CONJ:dKL}, one may prove effective versions of Putinar's Positivstellensatz for $\cube{n}$ starting from an effective version of Schm\"udgen's Positivstellensatz. In the original paper~\cite{deKlerkLaurent:hypercube2010}, de Klerk \& Laurent do so only for $d=2$. Magron~\cite{Magron:putinar} performs an analysis in the general case.
\begin{theorem}[{\cite[Thm. 4]{Magron:putinar}}] \label{THM:magron}
    Let $f \in \pos(\cube{n})$ be a polynomial of degree $d$. Assuming \Cref{CONJ:dKL} holds, we have:
    \[
        f \in \quadr(\cube{n})_{r}, \quad \text{ for } r \geq \exp(\frac{d^2n^{d+1}\cdot \|f\|_{\rm coef}}{\funcmin}),
    \]
    where, writing $f(x) = \sum_{\alpha} f_\alpha x^\alpha$, we set $\|f\|_{\rm coef} := \max_{\alpha} |f_\alpha| \cdot \frac{\prod_{i}(\alpha_i!)}{(\sum_i \alpha_i)!}$.
\end{theorem}
We note that the bound of \Cref{THM:magron} is asymptotically weaker than the general result of {Baldi \& Mourrain} (\Cref{THM:GeneralPutinar}), but it predates it, and its dependence on $n, d$ is more explicit. 

Our main result improves exponentially upon Magron's bound, with explicit constants, and without assuming \Cref{CONJ:dKL}. Compared to \Cref{COR:GeneralPutinar}, it improves the dependence on $\funcmax / \funcmin$ by a power of $10$. With respect to \Cref{THM:boxSchmudgen}, the degree bound is quadratically weaker, but it applies to representations in the quadratic module rather than the preordering.

\begin{restatable}[\Cref{THM:main} with explicit constants]{theorem}{putinarbox} 
\label{THM:main2}
Let $f \in \pos(\cube{n})$ be a polynomial of degree $d$ and denote $\funcmax$, $\funcmin$ the maximum and the minimum of $f$ on~$\cube{n}$, respectively. Then we have $f \in \quadr(\cube{n})_{rn}$ whenever
\[
    r \geq     
    4 \Chebyconst \cdot d^2(\log n) \cdot \frac{\funcmax}{\funcmin} + \max \left\{ \pi d \sqrt{2n}, \ \left(2 \Chebyconst \cdot \frac{\funcmax}{\funcmin} \cdot C(n, d)\right)^{1/2} \right\},
\]
where $\Chebyconst > 0$ is the absolute constant given in \Cref{LEM:chebybound} and $C(n, d)$ is the constant of \Cref{THM:boxSchmudgen}.
\end{restatable}

\subsection{Lower degree bounds}
To contextualize the positive results on the strength of sum-of-squares representations discussed above, it would be nice to have complementing \emph{negative} results, i.e, \emph{lower} bounds on the degree~$r$ required to represent positive polynomials. 
Remarkably, such results are rather rare in the literature%
\footnote{The exception is the case where $\cS(\mg) \subseteq \R^n$ is a finite set defined by polynomial equations, in which case \emph{every} nonnegative polynomial of degree $d$ on $\cS(\mg)$ has a representation in $\quadr(\mg)_N$ for some fixed $N = N(\mg, d) \in \N$. There is a large body of research in that setting, particularly when $\cS(\mg) \subseteq \{-1, 1\}^n$, see, e.g., \cite{KurpiszLeppanenMastrolilli:hardest} and references therein.}
For non-finite semialgebraic sets, the authors are only aware of the following result of Stengle~\cite{Stengle:negative}, which shows a lower degree bound already in the case $n=1$ if one uses a nonstandard representation of the interval $\cube{1} = [-1, 1] \subseteq \R$.
\begin{theorem}[{\cite[Thm.~4]{Stengle:negative}}]
\label{THM:Stengle}
For any $\epsilon > 0$ and $r \in \N$, we have: 
\[(1-x^2) + \epsilon \in 
\preo\big((1-x^2)^3\big)_r \implies
    r = \Omega(1/\sqrt{\epsilon}).
\]
\end{theorem}

Notably, the lower bound of \Cref{THM:Stengle} matches the best-known upper bound of \Cref{THM:boxSchmudgen} for the preordering of $\cube{n}$ (with the standard description).
In \Cref{SEC:negative}, we prove the following lower degree bound for the quadratic module:

\begin{restatable}{proposition}{negative} \label{PROP:negative}
For any $\epsilon > 0$ and $r \in \N$, we have 
\[ (1-x^2)(1-y^2) + \epsilon \in \quadr(\cube{2})_r \implies
r = \Omega(1/\sqrt[8]{\epsilon}).
\]
\end{restatable}
\Cref{PROP:negative}~differs from Stengle's result in three important ways.

First, it applies to a standard description of the hypercube~$\cube{n}$,
while \Cref{THM:Stengle} does not (see \Cref{SEC:Discussion} for a more detailed discussion). In particular, this description meets the constraint qualification conditions, see \Cref{def::CQC},
while the description that Stengle uses does not.

{
Second, notice that the Sch\"udgen-type representation of $(1-x^2)(1-y^2) + \epsilon \in \preo(\cube{n})$ is trivially of constant degree $4$ for all $\epsilon>0$, while its Putinar-type representation is proven to be of unbounded degree as ${\epsilon \to 0}$. Therefore, we have not only found a family of bounded-degree polynomials whose Putinar-type representations are of unbounded degree, but this family lies in $\preo(\cube{n})_4$. This is a significant difference with Stengle's result, as Putinar-type and Schm\"udgen-type representations coincide in \Cref{THM:Stengle}.
}

Third, the bound shown in our result is much weaker than Stengle's bound (it is of the order $1 / \sqrt[8]{\epsilon}$ compared to $1/\sqrt{\epsilon}$). In fact, Stengle~\cite{Stengle:negative} shows his bound is the best-possible up to log-factors, whereas we have no reason to believe our bound is close to optimal asymptotically (the upper bound of \Cref{THM:main} is of the order $1/\epsilon$).

\Cref{PROP:negative} generalizes to the setting $n > 2$ in a straightforward way, yielding an immediate implication for \Cref{CONJ:dKL}:
\begin{corollary} \label{COR:negative}
    Let $n \in \N$. For any $\epsilon > 0$ and $r \in \N$, we have:
    \[
    (1-x_1^2)(1-x_2^2) \ldots (1-x_n^2) + \epsilon \in \quadr(\cube{n})_r \implies r = \Omega(1/\sqrt[8]{\epsilon}).
    \]
    In particular, we have:
    \[
    (1-x_1^2)(1-x_2^2) \ldots (1-x_n^2) + \epsilon \in \quadr(\cube{n})_n \implies \epsilon = \Omega(1/n^8),
    \]
    for every $n \in \N$.
\end{corollary}
\begin{proof}
Suppose that $n \ge 2$ and we have a representation:
    \[
        \prod_{i=1}^n (1-x_i^2) + \epsilon = \sigma_0(\x) + \sum_{i=1}^n (1-x_i^2) \sigma_i(\x) \in \quadr(\cube{n})_r.
    \]
    Then setting $x_i=0$ for all $i > 2$ yields a representation:
    \[
        \prod_{i=1, 2} (1-x_i^2) + \epsilon = \sigma_0(\x) + \sum_{i=1, 2} (1-x_i^2) \sigma_i(x_1, x_2, \mathbf{0}) + \sum_{i=3}^n \sigma_i(x_1, x_2, \mathbf{0}) \in \quadr(\cube{2})_r,
    \]
    and so the lower bound $r = \Omega(1/\sqrt[8]{\epsilon})$ of \Cref{PROP:negative} applies here as well.
\end{proof}

In a more abstract direction, the existence of lower degree bounds for Putinar's and Schm\"udgen's Positivstellens\"atze is deeply related to the \emph{non}-stability property for $\quadr (\vb g)$ and $\preo (\vb g)$.
This connection is hardly found in the literature (with the exception of \cite{scheiderer_non-existence_2005}). In \Cref{SEC:Discussion},
 we therefore recall the notion of stability, give an overview of the related results and propose some research directions.

\subsection{Applications to polynomial optimization} \label{SEC:polyopt}
A polynomial optimization problem (POP) asks to minimize a given polynomial $p$ over a (compact) semialgebraic set $\cS(\mg)$, that is, to compute:
\begin{equation} \label{EQ:POP} \tag{POP}
    p_{\min} := \min_{\x \in \cS(\mg)} p(\x).
\end{equation}
Problems of the form~\eqref{EQ:POP} are generally hard and have broad applications~\cite{Lasserre:book,Laurent:polopt}. The simple case of the minimization of a polynomial on the unit hypercube is of particular interest. For example, the stability number of a graph $G = (V, E)$ equals (see for instance ~\cite[Eq.~(17)]{ParkHong:handelman})
\[
    \alpha(G) = \min_{\x \in [-1, 1]^V} \frac{1}{2} \sum_{i \in V} (1-x_i) - \frac{1}{4} \sum_{\{i, j\} \in E} (1-x_i) (1-x_j).
\]
The \emph{moment-SOS hierarchy}~\cite{Lasserre:lowbound, Parillo:thesis} provides a series of tractable \emph{lower bounds} on~$p_{\min}$. Namely, for $r \in \N$, we set:
\begin{equation} \label{EQ:lowerbound}
    p_{(r)} := \max_{\lambda \in \R} \big\{\lambda : p - \lambda \in \quadr(\mg)_r \big\} \leq p_{\min}.
\end{equation}
For fixed $r \in \N$, the bound $p_{(r)}$ may be computed by solving a semidefinite program of size polynomial in the number of variables $n$ and the number of constraints $m$ defining $\cS(\mg)$.
If $\quadr(\mg)$ is Archimedean, Putinar's Positivstellensatz tells us that $\lim_{r \to \infty} p_{(r)} = p_{\min}$, i.e., that the hierarchy converges. In this light, effective versions of Putinar's Positivstellensatz can be thought of as bounds on the \emph{rate} of this convergence. In this direction, our upper bound \Cref{THM:main} and our lower bound \Cref{THM:negative} imply the following.
\begin{corollary}\label{cor::upper}
    Let $p \in \R[\x]$ be a polynomial to be minimized over the hypercube~$\cube{n}$, defined by $g_i = 1-x_i^2$ for $i=1, \dots, n$, and let $p_{(r)} \leq p_{\min}$ be the lower bound of~\cref{EQ:lowerbound}. Then we have:
    \[
    p_{\min} - p_{(r)} = O(1/r) \quad (r \to \infty).
    \]
\end{corollary}
\begin{corollary}\label{cor::lower}
    For each $2\le n \in \N$, there exists a polynomial $p$ of degree $4$ to be minimized over the hypercube $\cube{n}$, defined by $g_i = 1-x_i^2$ for $i=1, \dots, n$, with ${p_{\min} = 0}$, $p_{\max} = 1$, and for which the bound of~\eqref{EQ:lowerbound} satisfies:
    \[
    p_{\min} - p_{(r)} = \Omega(1/r^8) \quad (r \to \infty).
    \]
\end{corollary}
In principle, one could define a (tighter) lower bound of the form~\eqref{EQ:lowerbound} using the \emph{preordering} $\preo(\mg)$ instead of the quadratic module~$\quadr(\mg)$. The analysis with the preordering is performed in \cite{LaurentSlot:hypercube} (see also \Cref{THM:boxSchmudgen}) where the authors deduce a convergence rate of $O(1/r^2)$. On the other hand, \Cref{cor::upper} shows weaker a degree bound in $O(1/r)$ for case of the quadratic module. But computing the bound using the preordering would require solving a semidefinite program that is not of polynomial size in the number of constraints $m$, while the bound using the quadratic module has linear size in $m$. For this reason, the bound of \Cref{cor::upper} in $O(1/r)$ is more relevant in practice, and its implications for polynomial optimization are arguably greater.

We notice also that the same asymptotic bounds hold true if we describe the hypercube $\cube{n}$ using the other standard set of inequalities, namely $1\pm x_i$ for $i \in [n]$, as explained in \Cref{SEC:Discussion}.

\section{Preliminaries}\label{SEC:prelim}

\subsection{Notations} 
\label{SEC:notations}
Throughout the article:
\begin{itemize}
    \item $[n] = \{ 1 , 2, \ldots, n\}$ for $n \in \N$;
    \item $x, t \in \R$ and $\x = (x_1, \ldots, x_n) \in \R^n$ denote real variables;
    \item $\R[\x] = \R[x_1, \ldots, x_n]$ denotes the polynomial ring in $n$ variables;
    \item $\Sigma[\x] \subseteq \R[\x]$ denotes the convex cone of sums of squares;
    \item {$\quadr(\cube{n})_r = \quadr(1-x_1^2, \dots , 1-x_n^2)_r$ is the truncated quadratic module at degree $r$ associated to the unit hypercube $\cube{n} = \cS(1-x_1^2, \dots ,1-x_n^2)$, consisting of polynomials of the form $\sigma_0 + \sigma_{1} (1-x_1^2)  + \ldots + \sigma_n (1-x_n^2)$ with $\sigma_i \in \Sigma[\x],~\mathrm{deg}(\sigma_0) \leq r$ and $\mathrm{deg}(\sigma_i(1-x_i^2)) \leq r$.}
    \item $f \in \R[\x]$ is a polynomial of degree $d$;
    \item $\funcmin, \funcmax$ are the minimum and maximum of $f$ on $\cube{n}$, respectively;
    \item for $k \in \N$ and $\x \in \R^n$, $\|{\x}\|_k = \big(\sum_{i=1}^n x_i^k\big)^{1/k}$ denotes the $L^k$-norm of $\x$, and $\|\x\|_\infty = \max_{i=1,\ldots,n} |x_i|$ denotes its $L^\infty$-norm.
\end{itemize}

\subsection{The Markov Brothers' inequality}
A key technical tool in the proofs of \Cref{SEC:positive} and~\Cref{SEC:negative} is the Markov Brothers' inequality \cite{MR1511855, markov1890ob}, see~\cite{Shadrin:Markov} for a modern account. 
In its general form, it bounds the norm of (higher-order) derivatives of a polynomial of given degree in terms of its supremum norm on an appropriate unit ball. It is  applied by Stengle~\cite{Stengle:negative} in his proof of \Cref{THM:Stengle}. To state the theorem, we first need to introduce Chebyshev polynomials.

\begin{definition}[see, e.g., \cite{Szego:book}]
    For $d \in \N$, the Chebyshev polynomial $\Cheby_d \in \R[x]$ of degree $d$ is defined as:
    \begin{equation} \label{EQ:chebydef}
        \Cheby_d(x) = 
        \begin{cases}
            \cos (d \arccos x) \quad & |x| \leq 1,\\
            \frac{1}{2} \bigg( \big( x + \sqrt{x^2 -1}\big)^d + \big( x - \sqrt{x^2 -1}\big)^d\bigg) \quad & |x| \geq 1.
        \end{cases}
    \end{equation}
\end{definition}
We recall that $|\Cheby_d(x)| \leq \Cheby_d(1) = 1$ for $x \in [-1, 1]$, that $\Cheby_d(x) = (-1)^d \cdot \Cheby_d(-x)$ for all $x \in \R$, and finally that $\Cheby_d(x)$ is monotonely increasing in $x$ for $x \geq 1$.

\begin{theorem}[{special case of \cite[Thm. 2]{Skalyga:markov}, see also ~\cite[Thm. 1]{Harris:markov}}]
\label{THM:genmarkov}
Let $\| \cdot \|$ be any norm on $\R^n$. Let $p \in \R[\x]$ be a polynomial of degree $d$, and write $\|p\|_\infty = \max_{\|\x\| \leq 1} |p(\x)|$. Then for all $k \geq 0$ and $\y \in \R^n$ with $\|\y\| \leq 1$, we have:
\begin{equation} \label{EQ:markovk>0}
    \left|\dv[k]{t} p(\x+t\y)\Bigr|_{t=0}\right| \leq 
    \begin{cases} 
        \|p\|_\infty \cdot T^{(k)}_d(1)     \quad & \|\x\| \leq 1, \\ 
        \|p\|_\infty \cdot T^{(k)}_d(\|\x\|) \quad & \|\x\| \geq 1.
    \end{cases}
\end{equation}
Here, $T_d^{(k)}$ is the $k$-th derivative of $T_d$. In particular, setting $k = 0$, we have:
\begin{equation} \label{EQ:markovk0}
    \abs{ p(\x) } \leq \|p\|_\infty \cdot T_d(\|\x\|) \quad \text{ for } \quad \|\x\| \geq 1.
\end{equation}
\end{theorem}
We will apply \Cref{THM:genmarkov} for the norm $\|\x\|_{\infty} = \max_{i=1, \ldots, n} |x_i|$, whose unit ball $\{ \x \in \R^n : \| \x \|_\infty \leq 1\}$ is the hypercube $\cube{n}$. The following lemma allows us to relate the supremum norm of polynomials on scaled unit balls (i.e., scaled hypercubes), which will be convenient in the proofs of our main results.
\begin{lemma}[{cf.~\cite[Eq. (3)]{Stengle:negative}}] \label{LEM:Markovscaled}
Let $\| \cdot \|$ be any norm on $\R^n$, and let $p \in \R[\x]$ be a polynomial of degree $d$. Then for any $\delta \in (0, 1)$, we have:
\[
    \max_{\|\x\|^2 \leq \frac{1}{1-\delta}} |p(\x)| \leq T_d \big(\frac{1}{1-\delta} \big) \cdot \max_{\|\x\|^2 \leq 1- \delta} |p(\x)|.
\]
\end{lemma}
\begin{proof}
    Using~\cref{EQ:markovk0}, we find that:
    \begin{align*}
        \max_{\|\x\|^2 \leq \frac{1}{1-\delta}} |p(\x)| &= \max_{\|\y\| \leq \frac{1}{1-\delta}} \big| p \big(\y \cdot \sqrt{1-\delta}\big)\big| \\
        &\leq \Cheby_d \big(\frac{1}{1-\delta}\big) \cdot \max_{\|\y\| \leq 1} \big| p \big(\y \cdot \sqrt{1-\delta}\big)\big| \\
        &= \Cheby_d \big(\frac{1}{1-\delta}\big) \cdot \max_{\|\x\|^2 \leq 1-\delta} |p(\x)|.
    \end{align*}
    To obtain the first equality, we simply change variables $\y = \x/\sqrt{1-\delta}$. Then, to get the inequality, we apply~\cref{EQ:markovk0} to the polynomial $\y \to p(\y \cdot \sqrt{1-\delta})$, noting that $\max_{|x| \leq \frac{1}{1-\delta}} \Cheby_d(x) = \Cheby\big(\frac{1}{1-\delta}\big)$. Finally, we change variables again to conclude. 
\end{proof}

In order to apply the inequalities stated above, we need the following facts on Chebyshev polynomials. These are known or follow quickly from results in the literature, but we restate them for ease of reference and completeness.
\begin{lemma}
\label{LEM:Markov} For any $k \geq 0$, we have:
    \[
        |T_d^{(k)}(x)| \leq 
        \begin{cases} 
            d^{2k} \quad & |x| \leq 1, \\
            d^{2k} \cdot |T_d(x)| & |x| \geq 1.
        \end{cases}
    \]
\end{lemma}
\begin{proof}
It is well known (see, e.g., \cite{Shadrin:Markov}) that, for any $k \geq 0$, we have
\[
\max_{-1 \leq x \leq 1} |T_d^{(k)}(x)| = T_d^{(k)}(1) = \frac{d^2 (d^2-1^2)\ldots(d^2-(k-1)^2)}{1\cdot3\cdot\ldots\cdot(2k-1)} \leq d^{2k}.
\]
Then, since $T_d^{(k)}$ is a polynomial of degree at most~$d$, we may apply~\eqref{EQ:markovk0} to find that
\[
    |T_d^{(k)}(x)| \leq \max_{-1 \leq y \leq 1} |T_d^{(k)}(y)| \cdot T_d(|x|) \leq d^{2k} \cdot |T_d(x)|, \quad \text{ when } |x| \geq 1. \qedhere
\]
\end{proof}

\begin{lemma}[cf.~{\cite[pf. of Thm.~4]{Stengle:negative}}] \label{LEM:chebybound}
Let $1 > \delta > 0$. Then, if $d = O(1/\sqrt{\delta})$ {and $\Cheby_d \in \R[x]$ is the Chebyshev polynomial of degree $d$}, we have:
\[
1 \leq T_d\big(\frac{1}{1-\delta} \big) = O(1) \quad (\delta \to 0).
\]
Furthermore, there exists an absolute constant $1 \le \Chebyconst \le e^5$ such that for any $d \geq 2$ and $\delta \leq 1/d^2$, we have $T_d\big(\frac{1}{1-\delta} \big) \leq \Chebyconst$.
\end{lemma}
\begin{proof}
    From~\cref{EQ:chebydef}, we find that for any $x \geq 1$:
    \begin{equation} \label{EQ:Chebyboundallx}
        \Cheby_d(x) \leq \big(x + \sqrt{x^2 - 1}\big)^d.
    \end{equation}
    As $\frac{1}{1-\delta} = 1 + \delta + O(\delta^2)$, we may use~\cref{EQ:Chebyboundallx} to get:
    \begin{align*}
        \Cheby_d(\frac{1}{1-\delta}) &\leq \big(1 + \delta + O(\delta^2) + \sqrt{1 + 2\delta + O(\delta^2) - 1}\big)^d \\
        &\leq \big(1 + O(\sqrt{\delta})\big)^d.
    \end{align*}
    It follows that $\Cheby_d(\frac{1}{1 - \delta}) = O(1)$ if $d = O(1/\sqrt{\delta})$. Now, if $d \geq 2$ and $\delta \leq 1/d^2$, we have $\frac{1}{1-\delta} \leq 1 + 2\delta$, and so by~\cref{EQ:Chebyboundallx} we get:
    \begin{align*}
        \Cheby_d(\frac{1}{1-\delta}) 
        &\leq \big(1 + 2\delta + \sqrt{(1 + 2\delta)^2 - 1}\big)^d. \\
        &\leq \big(1 + 2\delta + \sqrt{4\delta + 4\delta^2}\big)^d \\
        &\leq \big(1 + 5\sqrt{\delta}\big)^d \leq (1 + 5/d)^d \leq e^5. \qedhere
    \end{align*}
\end{proof}

\subsection{Schm\"udgen's Positivstellensatz for scaled hypercubes}
For our arguments in \Cref{SEC:positive}, we need an effective version of Schm\"udgen's Positivstellensatz for scaled hypercubes ${[-\eta, \eta]^n}$, with $\eta > 0$. \Cref{THM:boxSchmudgen} carries over to this setting in a straightforward way.
\begin{corollary}\label{cor::scled_schmudgen}
    For $\eta > 0$, write $\D = [-\eta, \eta]^n$. Let $f \in \R[\x]$ be a polynomial of degree $d$, and let $f_{\min, \D}, f_{\max, \D} > 0$ be the minimum and maximum of $f$ on $\D$, respectively. Then we have:
\[
    f \in \preo(\eta^2-x_1^2, \ldots, \eta^2-x_n^2)_{(r+1)n}, \quad \text{for }
     r \geq \max \bigg\{ \left( C(n, d) \cdot  \frac{f_{\max, \D}}{f_{\min, \D}} \right)^{1/2}, ~ \pi d \sqrt{2n} \bigg\}.
\]
Here, the constant $C(n,d)$ is the same as in \Cref{THM:boxSchmudgen}. 
\end{corollary}
\begin{proof}
Consider the polynomial $g(\x) = f(\eta \x)$, which is of degree $d$, and satisfies $g_{\min, [-1, 1]^n} = f_{\min, [-\eta, \eta]^n}$ and $g_{\max, [-1, 1]^n} = f_{\max, [-\eta, \eta]^n}$. We can apply \Cref{THM:boxSchmudgen} to write
\[
    f(\eta \x) = g(\x) = \sum_{I \subseteq [n]} \sigma_I(\x) \prod_{i \in I} (1-x_i^2)
\]
with appropriate degree bounds on the sums of squares $\sigma_I$. But then,
\begin{align*}
    f(\x) &= \sum_{I \subseteq [n]} \sigma_I(\x/\eta) \prod_{i \in I} (1-(x_i/\eta)^2) \\
    &= \sum_{I \subseteq [n]} \sigma_I(\x/\eta) \prod_{i \in I} \frac{1}{\eta^2}(\eta^2-x_i^2) \\
    &= \sum_{I \subseteq [n]} \eta^{-2|I|} \cdot \sigma_I(\x/\eta) \prod_{i \in I} (\eta^2-x_i^2),
\end{align*}
which is a decomposition of $f$ in $\preo(\eta^2 - x_1^2, \ldots, \eta^2 - x_n^2)$ of the desired degree.
\end{proof}

\section{Proof of the upper degree bound} \label{SEC:positive}

This section is dedicated to the proof of \Cref{THM:main} and \Cref{THM:main2}.

We start by recalling the technique used to prove general effective versions of Putinar's Positivstellensatz in~\cite{NieSchweighofer:putinar} and \cite{BaldiMourrain:putinar, BaldiMourrainParusinski:putinar}.
There, the authors reduce the question of representing a strictly positive polynomial~ $f$ on a general compact semialgebraic set $\cS(\vb g)$, to the question of representing strictly positive polynomials on a simpler compact domain $\D = \cS(\vb h) \supseteq \cS(\vb g)$.
More precisely, they construct a polynomial $p \in \quadr(\vb g)$ in such a way that $f- p > 0$ on~$\D$. 
As an effective version of Schm\"udgen's Positivstellensatz is available for the set $\D$, they then deduce that $f-p \in \preo(\vb h)$ (with an appropriate degree bound). Using the Archimedean hypothesis, we have $\preo(\vb h) \subseteq\quadr(\vb g)$, which gives the final representation $f = (f - p) + p \in \quadr(\vb g)$. The construction of the polynomial $p \in \quadr(\mg)$ and the effective Schm\"udgen's Positivstellensatz on $\D$ are the key parts of the proof: the different constructions in~\cite{NieSchweighofer:putinar} and~\cite{BaldiMourrain:putinar, BaldiMourrainParusinski:putinar} lead to an exponential and polynomial degree bound for the representation of $f \in \quadr(\mg)$, respectively. We refer to \cite{BaldiMourrain:putinar} for a more detailed list of references where this technique has been exploited.

\subsection{Overview of the proof}
Compared with  the general effective Putinar's Positivstellensatz, for the investigation of the special case $\cS(\vb g) = \cube{n}$ we make an important change of perspective: we consider a domain $\D$ that depends on $f$. Namely, we choose $\D$ to be a close enough outer approximation of $\cube{n}$, so that $f$ is not only strictly positive on $\cube{n}$, say $f \ge \funcmin > 0$ on $\cube{n}$, but also $f \ge \frac{1}{2}\funcmin > 0 $ on $\D$. In this way we can avoid using the perturbation polynomial $p$, and apply directly the representation results on the outer approximation $\D$. Concretely, we proceed as follows (see also \Cref{FIG:overview}).

\medskip\noindent\textbf{a. Selecting the outer domain.} We choose $\D = [-\eta, \eta]^n$ to be a scaled hypercube containing $\cube{n}$, where $\eta > 1$ will be chosen in such a way that:
\[
\min_{\x \in \D} f(\x) \geq \frac{1}{2} \min_{\x \in \cube{n}} f(\x) > 0,
\]
see~\Cref{lem::choice_of_epsilon}.

\medskip\noindent\textbf{b. Obtaining a Schm\"udgen-type representation} We then apply \Cref{cor::scled_schmudgen}, a scaled version of \Cref{THM:boxSchmudgen} on $\D =\cS(\eta^2 - x_1^2, \dots, \eta^2 - x_n^2)$, to represent $f$ as an element of the preordering $\preo(\eta^2 - x_1^2, \dots, \eta^2 - x_n^2)$, with appropriate degree bounds.

\medskip\noindent\textbf{c. Lifting the representation}
Finally, we lift the representation of $f$ from the preordering $\preo(\eta^2 - x_1^2, \dots, \eta^2 - x_n^2)$ to the quadratic module $\quadr(1-x_1^2, \dots, 1-x_n^2)$.
For this purpose, we make use of the metric balls:
\[
    \{ \x \in \R^n : n - \metric{\x}{2q} \geq 0 \}, \quad \text{ where } \metric{\x}{2q} = x_1^{2q} + \ldots + x_n^{2q}   \quad (q \in \N).
\]
Choosing $q \in \N$ large enough so that $\eta \geq \sqrt[2q]{n}$, we show in \Cref{lem::basic2} and \Cref{thm::from_preordering_to_module} that:
\begin{align*}
    \preo(\eta^2 - x_1^2, \dots, \eta^2 - x_n^2) 
    &\subseteq \preo(n - \metric{\x}{2q}) \\
    &= \quadr(n - \metric{\x}{2q}) \\
    &\subseteq \quadr(1-x_1^2, \dots, 1-x_n^2) = \quadr (\cube{n})
\end{align*}
with appropriate degree bounds for the truncated versions. Using the Schm\"udgen-type representation obtained in the previous step, this will give us a Putinar-type representation with appropriate degree bounds:
\[
    f \in \preo(\eta^2 - x_1^2, \dots, \eta^2 - x_n^2) \subseteq\quadr(\cube{n}).
\]

\begin{figure}
    \centering
    \begin{tikzpicture}[scale=1.5]


\draw[thick, fill=gray, fill opacity=0.2] (-1,-1) -- (1,-1) -- (1,1) -- (-1,1) -- (-1,-1);  

\draw[thick, gray] (-1.75, -1.75) -- (1.75, -1.75);
\draw[thick, gray] (-1.75, -1.75) -- (-1.75, 1.75);

\draw (0, -1.72) -- (0, -1.78) node[below] {$0$}; 
\draw (-1.72, 0) -- (-1.78, 0) node[left] {$0$}; 
\draw (-1.42, -1.72) -- (-1.42, -1.78) node[below] {$-\sqrt{2}$}; 
\draw (1.42, -1.72) -- (1.42, -1.78) node[below] {$\sqrt{2}$}; 
\draw (-1.72, -1.42) -- (-1.78, -1.42) node[left] {$-\sqrt{2}$}; 
\draw (-1.72, 1.42) -- (-1.78, 1.42) node[left] {$\sqrt{2}$}; 


\begin{scope}[scale=1.43]
    \draw[thick, fill=gray, fill opacity=0.2] (-1,-1) -- (1,-1) -- (1,1) -- (-1,1) -- (-1,-1);
    \draw[thick, dashed, fill=gray, fill opacity=0.2] (0,0) circle [radius=1cm];
\end{scope}

\draw[<->] (-1.41, 1.52) -- (0, 1.52) node [midway, above] {$\eta = \sqrt{2}$};

\begin{scope}[xshift=120]
\draw[thick, fill=gray, fill opacity=0.2] (-1,-1) -- (1,-1) -- (1,1) -- (-1,1) -- (-1,-1);

\draw[thick, gray] (-1.75, -1.75) -- (1.75, -1.75);
\draw[thick, gray] (-1.75, -1.75) -- (-1.75, 1.75);

\draw (0, -1.72) -- (0, -1.78) node[below] {$0$}; 
\draw (-1.72, 0) -- (-1.78, 0) node[left] {$0$};
\draw (-1.22, -1.72) -- (-1.22, -1.78) node[below] {$-\sqrt[4]{2}$}; 
\draw (1.22, -1.72) -- (1.22, -1.78) node[below] {$\sqrt[4]{2}$}; 
\draw (-1.72, -1.22) -- (-1.78, -1.22) node[left] {$-\sqrt[4]{2}$}; 
\draw (-1.72, 1.22) -- (-1.78, 1.22) node[left] {$\sqrt[4]{2}$}; 

\draw [dashed, thick, fill=gray, fill opacity=0.2] plot [smooth cycle] coordinates {(-1,1) (1,1) (1,-1) (-1,-1)};


\begin{scope}[scale=1.22]
    \draw[thick, fill=gray, fill opacity=0.2] (-1,-1) -- (1,-1) -- (1,1) -- (-1,1) -- (-1,-1);
\end{scope}

\draw[<->] (-1.21, 1.32) -- (0, 1.32) node [midway, above] {$\eta = \sqrt[4]{2}$};
\end{scope}


\end{tikzpicture}
    \captionsetup{singlelinecheck=off}
    \caption[.]{Example of the key regions involved in the proof of \Cref{THM:main} for $n=2$ and $q = 1, 2$ (on the left, right, respectively). The shaded regions depict the sets:
    \begin{displaymath}    
    \cube{n} = [-1, 1]^n \subseteq \cS(n-\metric{\x}{2q}) \subseteq [-\eta, \eta]^n = \D. 
    \end{displaymath} \vspace{-2em}
    }
\label{FIG:overview}
\end{figure}
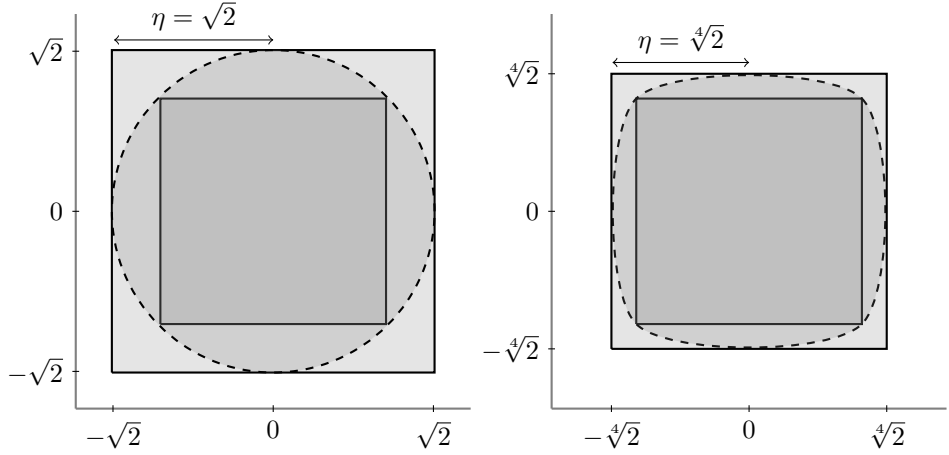

\subsection{Proof of \Cref{THM:main}}
To present and describe our proof in a compact way, we introduce the following definition.
\begin{definition}
    Let $\quadr(\mg)$ be a (finitely generated) quadratic module.
    We say that a tuple of polynomials $\vb h = (h_1, \dots, h_s) \subseteq \quadr(\mg)$ has \emph{degree shift} $\ell$ with respect to $\vb g$ if $h_i \in \quadr(\vb g)_{\deg h_i +\ell}$ for all $i \in [s]$.
\end{definition}

\pagebreak
We will make use of the following elementary lemma, that we state in general for future reference.
\begin{lemma}[Degree shift]\label{lem::degree_shift}
Let $\vb h = (h_1, \dots, h_s) \subseteq \quadr(\mg)$ be a tuple of polynomials with degree shift $\ell$ w.r.t. $\mg$.
Then, for all $d \in \N$,
\begin{enumerate}[label=(\roman*)]
    \item $\quadr(\vb h)_d \subseteq\quadr(\vb g)_{d+\ell}$;
    \item $\preo(\vb h)_{d} \subseteq\quadr(g_1)_{d + s \ell}$ if $\vb g = (g_1)$ consists of a single polynomial.
\end{enumerate}
\end{lemma}
\begin{proof}
    {The proof is simply by substitution and tracking of degrees.}

    {We set $h_0 = g_0=1$ for notational convenience, and we start by proving the first part. Let $q = \sum_{k=0}^r \overline{\sigma}_{k}h_k \in \quadr(\vb h)_d$: we want to show that $q \in \quadr(\vb g)_{d + \ell}$. By assumption, there exists a representation $h_i = \sum_{j=0}^m \sigma_{i,j} g_j \in \quadr(\vb g)_{\deg h_i +\ell}$. Notice that, by definition, $\deg (\sigma_{i,j} g_j) = \deg \sigma_{i,j} + \deg g_j \le \deg h_i +\ell$ and $\deg (\overline{\sigma}_{k}h_k )= \overline{\sigma}_{k}+\deg h_k \le d$ for all $i,j,k$. Therefore:
    \[
        \deg (\overline{\sigma}_{k} \sigma_{k,j} g_j) = \deg \overline{\sigma}_{k} + \deg \sigma_{k,j} + \deg  g_j \le d+\ell
    \]
    and finally
    \[
        q = \sum_{k=0}^r \overline{\sigma}_{k}h_k = \sum_{j=0}^m \left(\sum_{k=0}^r \overline{\sigma}_{k} \sigma_{k,j}\right) g_j \in \quadr(\vb g)_{d + \ell}
    \]
    concluding the proof of the first part.}

    {For the second part, we proceed in a similar way. Let $p = \sum_{I \subseteq [m]} \overline{\sigma}_{I} h_I \in \preo(\vb h)_{d}$: we want to show that $p \in \preo(g_1)_{d + s \ell}$. Choose a representation $h_i = \sigma_{0,i} + \sigma_{1,i} g_1 \in \quadr(g_1)_{\deg h_i + \ell}$, and notice that, by definition, $\deg \overline{\sigma}_{I} \le d - \deg h_I$ and
    \[
        \prod_{i\in I} (\sigma_{0,i} + \sigma_{1,i} g_1) \in \preo(g_1)_{\sum_{i \in I}\left( \deg h_i + \ell \right)} = \preo(g_1)_{\deg h_I + \abs{I}\ell} \subset \preo(g_1)_{\deg h_I + s \ell}.
    \]
    Therefore,
    \[
        \overline{\sigma}_{I} \prod_{i\in I} (\sigma_{0,i} + \sigma_{1,i} g_1) \in \preo(\vb g)_{\deg h_I + s \ell + \deg \overline{\sigma}_{I}} \subseteq\preo(g_1)_{\deg h_I + s \ell + d - \deg h_I} = \preo(g_1)_{d + s \ell}
    \]
    and finally:
    \[
        p = \sum_{I \subseteq [r]} \sigma_{I} h_I = \sum_{I \subseteq [r]} \sigma_{I} \prod_{i\in I} h_i = \sum_{I \subseteq [r]} \sigma_{I} \prod_{i\in I} (\sigma_{0,i} + \sigma_{1,i} g_1) \in \preo(g_1)_{d + s \ell}. \qedhere
    \]}
\end{proof}

In the following, we will apply \Cref{lem::degree_shift} two times.
First, to lift the representation of $f$ from $\preo(\eta^2 - x_1^2, \ldots, \eta^2 - x_n^2)$ to $\preo(n - \metric{\x}{2q}) = \quadr(n - \metric{\x}{2q})$. There, we apply the second part of \Cref{lem::degree_shift} with $g_1 = n - \metric{\x}{2q}$ and $h_i = \eta^2 - x_i^2$ for $i \in [n]$, see \Cref{lem::basic2}.

Second, to lift the representation of $f$ from $\quadr(n - \metric{\x}{2q})$ to $\quadr(1-x_1^2, \dots, 1-x_n^2)$. Here, we use the first part of \Cref{lem::degree_shift} with $\vb h = h_1 = n - \metric{\x}{2q}$ and $g_i = 1-x_i^2$ for $i \in [n]$, see the proof of \Cref{thm::from_preordering_to_module}. 

For these applications, we will need to determine two numbers (degree shifts) $\ell_1, \ell_2 \in \N$ depending on $\eta > 0$ such that:
\begin{align}
\eta^2 - x_1^2, \ldots, \eta^2 - x_n^2 &\in \preo(n - \metric{\x}{2q})_{2+\ell_1}, \\
n - \metric{\x}{2q} &\in \quadr(1-x_1^2, \dots, 1-x_n^2)_{2q + \ell_2}.
\end{align}%

To determine these degree shifts, we start by investigating the univariate case.

\begin{lemma}\label{lem::basic}
    For all $q \in \N$, the {(optimal)} degree shift of $1-x^2$ with respect to $1-x^{2q}$ is equal to $2q-2$, i.e.
    $1 - x^2 \in \quadr(1-x^{2q})_{2q}$. 
\end{lemma}
\begin{proof}
    For all $1 \le q \in \N$, consider the identity:
    \begin{equation}\label{eq::direct}
        1 - x^2 = \frac{(q-1) - qx^2 + x^{2q}}{q} + \frac{1-x^{2q}}{q}
    \end{equation}
    Notice that $(q-1) - qx^2 + x^{2q} \in \nnon(\R)$,
    since the polynomial has minimum equal to $0$ attained at $\pm 1$. Moreover, sums of squares and nonnegative polynomials coincide in one variable, and thus $(q-1) - qx^2 + x^{2q} \in \Sigma[x]_{2q}$. Therefore, \cref{eq::direct} implies that $1 - x^2 \in \Sigma[x]_{2q} + \R_{\ge 0} (1-x^{2q}) =\quadr(1-x^{2q})_{2q}$. {Since $1-x^2$ is not globally nonnegative on $\R$, no representation of smaller degree is possible. This concludes the proof}.
\end{proof}
We refer to \Cref{app::degree_shift}, and in particular to \cref{eq::sos}, for a more detailed discussion of \Cref{lem::basic} and \cref{eq::direct}.

We turn our attention to the multivariate case. We investigate the degree shift of $\eta^2 - x_1^2, \dots , \eta^2 - x_n^2$, i.e. the polynomials defining a scaled hypercube containing $\cube{n}$ for $\eta \ge 1$, with respect to $n - \metric{\x}{2q}$, the polynomial defining the $L^{2q}$ unit ball. Recall that the parameter $\eta$ will be chosen in such a way $f \ge \funcmin > 0$ on $\cube{n}$ implies $f \ge \frac{1}{2}\funcmin > 0$ on $[-\eta, \eta]^n = \cS(\eta^2 - x_1^2, \dots, \eta^2 - x_n^2)$,
see \Cref{lem::choice_of_epsilon}.

We prove that the degree shift of $\eta^2 - x_1^2, \dots , \eta^2 - x_n^2$ w.r.t. $n - \metric{\x}{2q}$ coincides with the one of \Cref{lem::basic}.

\begin{lemma} \label{lem::basic2}
    Let $\eta = \sqrt[2q]{n}$. Then the degree shift of {$(\eta^2 - x_1^2, \dots , \eta^2 - x_n^2)$}
    with respect to $n - \metric{\x}{2q}$ is $2q-2$. In other words, $\eta^2 - x_i^2 \in \quadr(n - \metric{\x}{2q})_{2q}$ for all $i \in [n]$.
\end{lemma}
\begin{proof}
    First, notice that
    \[
        1-x_i^{2q} = \sum_{j\neq i} x_j^{2q} + 1 - \metric{\x}{2q} \in \quadr(1 - \metric{\x}{2q})_{2q}
    \]
    and thus from \Cref{lem::basic} we deduce that 
    \[ 
    1 - x_i^2 \in \quadr(1 - \metric{\x}{2q})_{2q}
    \]
    for all $i$.

    Now let $\eta = \sqrt[2q]{n}$. If we substitute $x_i \mapsto x_i/\eta$ in $1-x_i^2$
    we obtain $\frac{\eta^2 - x_i^2}{\eta^2}$, while if we substitute in $1 -\metric{\x}{2q}$
    we obtain $\frac{n - \metric{\x}{2q}}{n}$.
    Making these substitutions in the expression $1 - x_i^2 \in \quadr(1 - \metric{\x}{2q})_{2^m}$
    we therefore see that
    \[\eta^2 - x_i^2 \in \quadr(n - \metric{\x}{2q})_{2q}\]
    concluding the proof.
\end{proof}

We refer to \Cref{app::degree_shift} and in particular to \cref{eq::deformation} for a more detailed discussion of \Cref{lem::basic2}. {The choice of $q$ (or equivalently, the choice of $\eta$), will be a key step for the proof of \Cref{THM:main2}. This choice will be governed by \Cref{lem::choice_of_epsilon} (where we write $\eta = 1+\epsilon$).}

We are now ready to show one of the main results of this section.

\begin{theorem}\label{thm::from_preordering_to_module}
    Let $\eta = \sqrt[2q]{n}$. Then for all $k \in \N$:
    \[
        \preo(\eta^2 - x_1^2, \dots , \eta^2 - x_n^2)_{k} \subseteq\quadr(1 - x_1^2, \dots , 1 - x_n^2)_{k + n (2q-2)}
    \]
\end{theorem}
\begin{proof}
    We start moving from $\preo(\eta^2 - x_1^2, \dots , \eta^2 - x_i^2)$ 
    to $\quadr(n - \metric{\x}{2q})$. From \Cref{lem::basic2}, the degree shift {of $(\eta^2 - x_1^2, \dots , \eta^2 - x_i^2)$ with respect to $n - \metric{\x}{2q}$} is $\ell_1 = 2q-2$, and from \Cref{lem::degree_shift}(ii) we have
    \[
    \preo(\eta^2 - x_1^2, \dots , \eta^2 - x_i^2)_{k} \subseteq  \quadr(n - \metric{\x}{2q})_{k+ n \ell_1} 
    \]
    
    We now move from $\quadr(n - \metric{\x}{2q})$ to $\quadr(1 - x_1^2, \dots , 1 - x_n^2)$. Notice that, since $1 - x_i^{2q} = (1-x_i^2)(1+x_i^2 + \dots + x_i^{2q-2})$, we have:
    \[
        n - \metric{\x}{2q} = \sum_{i=1}^n 1 - x_i^{2q} \in \quadr(1 - x_1^2, \dots , 1 - x_n^2)_{2q}
    \]
    and thus the degree shift of $n - \metric{\x}{2q}$ with respect to $1 - x_1^2, \dots , 1 - x_n^2$ is equal to ${\ell_2 = 0}$. From \Cref{lem::degree_shift}(i) we then deduce that 
    \[
    \quadr(n - \metric{\x}{2q})_{k+ n \ell_1} \subseteq \quadr({1 - x_1^2}, \dots, {1 - x_n^2})_{k+ n \ell_1}.
    \]
    We therefore have the chain of inclusions:
    \begin{align*}
        \preo(\eta^2 - x_1^2, \dots , \eta^2 - x_i^2)_{k} 
        \subseteq \quadr(n - \metric{\x}{2q})_{k+ n \ell_1} \subseteq\quadr(1 - x_1^2, \dots , 1 - x_n^2)_{k+ n \ell_1},
    \end{align*}
    where $l_1 = 2q-2$, concluding the proof.
\end{proof}

\Cref{thm::from_preordering_to_module} allows to shift the representation of a polynomial in the preordering of $[-\eta, \eta]^n$ to a representation in the quadratic module of $\cube{n}$. We refer to \Cref{app::degree_shift} and in particular \cref{eq::explicit_preo_module} for explicit expressions leading to this inclusion.

We now write $\eta = 1 + \epsilon$ and bound $\epsilon$ in such a way $f \ge \funcmin > 0$ on $\cube{n}$ implies $f \ge \frac{1}{2}\funcmin > 0$ on $[-1-\epsilon, 1+\epsilon]^n = [-\eta, \eta]^n$.

\begin{lemma}\label{lem::choice_of_epsilon}
    Let $f \ge \funcmin > 0$ on $\cube{n}$ be a polynomial of degree $d$.
    Let $\Chebyconst \geq 1$ be the absolute constant of \Cref{LEM:chebybound}. Then $f \ge \frac{1}{2}\funcmin$ on $[-1-\epsilon, 1+\epsilon]^n$
    whenever 
    \begin{equation} \label{EQ:epsilon}
    \epsilon \le \frac{f_{\min}}{2 \Chebyconst \cdot d^2 \cdot \funcmax}.
    \end{equation}
\end{lemma}
\begin{proof}
    {Assume that $\epsilon$ is as in \cref{EQ:epsilon} and let $\z \in [-1-\epsilon, 1+\epsilon]^n$ be a minimizer of $f$ on $[-1-\epsilon, 1+\epsilon]^n$, i.e. $f(\z) = f_{\min, [-1-\epsilon, 1+\epsilon]^n}$. Clearly,  $f_{\min, [-1-\epsilon, 1+\epsilon]^n} \leq \funcmin$. If $\z \in \cube{n}$, then $f(\z) = \funcmin \geq \frac{1}{2}\funcmin$ and there is nothing to prove. So assume $\z \notin \cube{n}$ and let $\widehat\z \in \cube{n}$ be a point in $\cube{n}$ with $0 < \|\z-\widehat\z\|_\infty \leq \epsilon$.
    Consider the univariate polynomial $F$ given by:}    
    \[
        F(u) := f(\widehat\z + u \cdot \vb v), \quad \text{ where } \vb v := \frac{(\z-\widehat\z)}{\|\z-\widehat\z\|_\infty}.
    \]
    Note that $F(0) = f(\widehat\z)$ and $F(\|\z-\widehat\z\|_\infty) = f(\z)$. We now bound the derivative $F'(u)$ of $F$ for all $0 \leq u \leq \|\z-\widehat\z\|_\infty$, so that we can obtain a bound on the difference $|f(\z) - f(\widehat\z)|$.   
    First, notice that
    \begin{equation}\label{eq::F_prime}
        F'(u) = \dv{t} f\big((\widehat\z + u \cdot \vb v ) + t \cdot \vb v \big)|_{t=0}.    
    \end{equation}
    As $\| \vb v\|_{\infty} \leq 1$, we can apply~\cref{EQ:markovk>0} to the polynomial~$f$, with $\x = \widehat \z + u \cdot \vb v$ and $\y = \vb v$, to get
    \[
        |F'(u)| = \left|\dv{t} f\big((\widehat\z + u \cdot \vb v ) + t \cdot \vb v \big)|_{t=0}\right| \leq T'_d(\norm{\widehat \z + u \cdot \vb v}_\infty ) \cdot \max_{\|\x\|_{\infty}
        \leq 1 } |f(\x)|.
    \]
    Notice that, since $0 \le u \le \|\z-\widehat\z\|_\infty \leq \epsilon$, we have \[ \norm{\widehat \z + u \cdot \vb v}_\infty \leq 1 + u \leq 1 + \epsilon. \] Now, using \Cref{LEM:Markov} and monotonicity of $T_d$ we get
    \[
        T'_d(\norm{\widehat \z + u \cdot \vb v}_\infty ) \leq d^2 \cdot T_d(\norm{\widehat \z + u \cdot \vb v}_\infty ) \leq d^2 \cdot T_d(1+u) \leq d^2 \cdot T_d(1+\epsilon).
    \]
    Finally, noting that $\max_{\|\x\|_{\infty} \leq 1 } |f(\x)| = \funcmax$, we may conclude that:
    \begin{equation} \label{EQ:Fprimesmall}
        |F'(u)| \leq d^2 \cdot T_d(1+\epsilon) \cdot \funcmax \quad (0 \leq u \leq \|\z-\widehat\z\|_\infty).
    \end{equation}
    Assuming \cref{EQ:epsilon}, we have $\epsilon \leq 1/d^2$ ({as $\funcmin/\funcmax \leq 1$}). Therefore, $\Cheby_d(1+\epsilon) \leq \Cheby_d(\frac{1}{1 - \epsilon}) \leq  \Chebyconst$, where $\Chebyconst \ge 1$ is the absolute constant of \Cref{LEM:chebybound}. Using \cref{EQ:Fprimesmall} and the fact that $\|\z-\widehat\z\|_\infty \leq \epsilon$, we thus have:
    \begin{align*} 
    |f(\z) - f(\widehat\z)| &= |F(0) - F(\|\z-\widehat\z\|_\infty)| \\
    &\leq \|\z-\widehat\z\|_\infty \cdot \max_{0 \leq u \leq \|\z-\widehat\z\|_\infty} |F'(u)| \\
    &\leq \epsilon \cdot d^2 \cdot T_d(1+\epsilon) \cdot \funcmax  \leq  \epsilon  \cdot  d^2 \cdot \Chebyconst \cdot  \funcmax.
    \end{align*}
    In conclusion, if we choose $\epsilon$ as in~\cref{EQ:epsilon}, we have:
    \[
    f_{\min, [-1-\epsilon, 1+\epsilon]^n} = f(\z) \geq f(\widehat \z) - \Chebyconst \cdot \epsilon \cdot d^2 \cdot \funcmax \geq \funcmin - \frac{1}{2}\funcmin = \frac{1}{2}\funcmin. \qedhere
    \]
\end{proof}

We are ready to prove our main result.

\putinarbox*
\begin{proof}[Proof of \Cref{THM:main} and \Cref{THM:main2}]

    Let $ 0 < \epsilon = \frac{f_{\min}}{2 \Chebyconst \cdot d^2 \cdot \funcmax}$ be as in~\cref{EQ:epsilon}, and let $q \in \N$ be the smallest integer such that:
    \begin{equation}
    \label{eq::linear_dependence}
        2q \ge \frac{2\log {n}}{\epsilon} = 4 \Chebyconst \cdot (\log n) \cdot d^2 \cdot \frac{\funcmax}{\funcmin}.
    \end{equation}
    Then, as $\epsilon \leq 1$, we have $\log (1 + \epsilon) \geq \epsilon - \frac{1}{2}\epsilon^2 \geq \frac{1}{2} \epsilon$,
    and thus
    \[
        \frac{2q \cdot \log(1+\epsilon)}{\log n} \geq 1,
    \]
    or in other words, we have $\sqrt[2q]{n} \le 1+\epsilon$. {Therefore, if we set $\eta = \sqrt[2q]{n}$, we have $\eta \leq 1 + \epsilon$, and
    we can deduce from \Cref{lem::choice_of_epsilon} that $f \ge \frac{1}{2}\funcmin$ on $[-\eta, \eta]^n$.}
    From \Cref{cor::scled_schmudgen}, we have a representation $f \in \preo(\eta^2-x_1^2, \ldots, \eta^2-x_n^2)_{(\ell+1)n}$ if $\ell$ is any integer such that
    \[
        \ell \ge \max \left\{ \pi d \sqrt{2n}, \ \left(\frac{f_{\max, [-\eta, \eta]^n}}{f_{\min, [-\eta, \eta]^n}} \cdot C(n, d)\right)^{1/2} \right\}
    \]
    We want to express the above bound using $\frac{\funcmax}{\funcmin}$ instead of $\frac{f_{\max, [-\eta, \eta]^n}}{f_{\min, [-\eta, \eta]^n}}$. For this, recall first that $f_{\min, [-\eta, \eta]^n} \ge \frac{1}{2}\funcmin$ by construction. Second, since $\epsilon \leq 1/d^2$, we have $\Cheby_d(\eta) \leq \Cheby_d(1+\epsilon) \leq \Cheby_d(\frac{1}{1 - \epsilon}) \leq \Chebyconst$ by~\Cref{LEM:chebybound} and we can use~\cref{EQ:markovk0} to get: 
    \[
        f_{\max, [-\eta, \eta]^n} \leq \Cheby_d(\eta) \cdot \funcmax \leq \Chebyconst \cdot \funcmax.
    \]
    Therefore, we have:
        \begin{align}
    \left(\frac{f_{\max, [-\eta, \eta]^n}}{f_{\min, [-\eta, \eta]^n}} \cdot C(n, d)\right)^{1/2} \leq \left(2 \Chebyconst \cdot \frac{\funcmax}{\funcmin} \cdot C(n, d)\right)^{1/2},
    \end{align}
    and we can thus choose $\ell$ as the smallest integer such that:
    \[
        \ell \ge \max \left\{ \pi d \sqrt{2n}, \ \left(2 \Chebyconst \cdot \frac{\funcmax}{\funcmin} \cdot C(n, d)\right)^{1/2} \right\}.
    \]    
    To conclude the proof, we apply \Cref{thm::from_preordering_to_module} and deduce that:
    \begin{align*}
        f \in \preo(\eta^2-x_1^2, \ldots, \eta^2-x_n^2)_{(\ell+1)n}
        & \subseteq\quadr(1-x_1^2, \ldots, 1-x_n^2)_{n (\ell + 1) + n(2q-2)} \\
        & = \quadr(1-x_1^2, \ldots, 1-x_n^2)_{n (2q + \ell - 1)}.
    \end{align*}
    Since $q$ is the smallest integer satisfying~\cref{eq::linear_dependence}, we have $f \in \quadr(1-x_1^2, \ldots, 1-x_n^2)_{rn}$ whenever
    \[
    r \ge 4 \Chebyconst \cdot (\log n) \cdot d^2 \cdot \frac{\funcmax}{\funcmin} + \ell \ge 2q+\ell-1. \qedhere
    \]
    \end{proof}

\section{Proof of the lower degree bound}
\label{SEC:negative}

In this section, we prove our lower degree bound, \Cref{THM:negative}. We consider the bivariate polynomial:
\[
    f(x, y) = (1-x^2)(1-y^2).
\]
Clearly, $f$ is nonnegative on $\cube{2}$ and $f \in \preo(\cube{2})_4$. On the other hand, $f \not \in \quadr(\cube{2})$.
This is well known, but we give an analytical argument for this fact as a warmup to the proof of \Cref{PROP:negative}.
\begin{proposition}\label{prop::basic}
We have $f(x,y) = (1-x^2)(1-y^2) \not\in \quadr(\cube{2})$.
\end{proposition}
\begin{proof}
Suppose that $f \in \quadr(\cube{2})$. Then $f$ can be written as:
\begin{equation} \label{EQ:fdecomp}
    f(x, y) = \sigma_0(x, y) + (1-x^2)\sigma_1(x, y) + (1-y^2)\sigma_2(x, y),
\end{equation}
where the $\sigma_i \in \Sigma[x, y]$ are sums of squares (in particular globally nonnegative). Note that $f(1, 1) = 0$. We can conclude immediately that $\sigma_0(1, 1) = 0$. In fact, we have that $\sigma_i(1, 1) = 0$ for all $i \in \{0,1, 2\}$. Indeed, suppose for instance that $\sigma_1(1, 1) > 0$. Then there exists an $1 \geq \varepsilon > 0$ such that $\sigma_1(\sqrt{1 - \varepsilon},\ 1) > 0$ by continuity. But this leads to the contradiction:
\[
    0 = f(\sqrt{1 - \varepsilon},\ 1) \geq  \epsilon \cdot \sigma_1(\sqrt{1 - \varepsilon},\ 1) > 0.
\]

To finish the argument, note that from the definition of $f$,
\begin{equation} \label{EQ:fsecondderiv}
\dv[2]{t} f(1 + t, 1-t)|_{t=0} < 0.
\end{equation}
As $\sigma_0, \sigma_1, \sigma_2$ are globally nonnegative, and since $\sigma_i(1, 1) = 0$, we have that:
\begin{align*}
    \dv{t} \sigma_i(1 + t, 1-t)|_{t=0} &= 0, \\
    \dv[2]{t} \sigma_i(1 + t, 1-t)|_{t=0} &\geq 0.
\end{align*}
By~\cref{EQ:fdecomp}, this would imply that
$\dv[2]{t} f(1 + t, 1-t)|_{t=0} \geq 0$, contradicting~\eqref{EQ:fsecondderiv}.
\end{proof}

The idea for the proof of \Cref{PROP:negative} (and thus of \Cref{THM:negative})
is to transform the proof above into a quantitative result. This resembles the argument of Stengle~\cite{Stengle:negative}.
\negative*
\begin{proof}
Let $f(x,y) = (1-x^2)(1-y^2)$, and suppose that $f + \epsilon \in \quadr(\mg)_r$, i.e. that we have a decomposition:
\begin{equation}
    \label{EQ:decomp}
    (1-x^2)(1-y^2) + \epsilon = \sigma_0(x,y) + (1-x^2)\sigma_1(x,y) + (1-y^2)\sigma_2(x,y),
\end{equation}
where $\sigma_0, \sigma_1, \sigma_2$ are sums of squares of polynomials of degree $\mathrm{deg}(\sigma_i) \leq r$ (more precisely, we have $\mathrm{deg}(\sigma_i) \leq r-2$ for $i=1,2$, but this will not be important). 
We consider the situation locally around the point $(1, 1) \in \cube{2}$. We can deduce the following facts.
\begin{fact} \label{FACT:sigma1small}
We have $\sigma_1(1, 1) \leq 
\frac{1}{2} \epsilon r^2$.
\end{fact}
\begin{proof}
Consider the univariate polynomial $p(x) = (1-x^2)\sigma_1(x, 1)$. By~\cref{EQ:decomp}, we have $0 \leq p(x) \leq \epsilon$ for $x \in [-1, 1]$. By \Cref{THM:genmarkov} and \Cref{LEM:Markov}, we find $|p'(x)| \leq \epsilon r^2$ for $x \in [-1, 1]$.
Setting $x=1$, we thus have:
\[
    \epsilon r^2 \geq |p'(1)| = 2 \sigma_1(1, 1). \qedhere
\]
\end{proof}

\begin{fact}
\label{FACT:sigmasmall}
For any $1 > \delta \geq \epsilon$, we have:
\[
    \sigma_1(x, y) \leq 2 \cdot T_r\big(\frac{1}{1-\delta}\big) \quad \text{ for } x^2 \leq \frac{1}{1 - \delta},~y^2 \leq \frac{1}{1 - \delta}.
\]
In particular, 
\begin{equation} \label{EQ:sigma1small}
\max_{x, y \in [-1, 1]} \sigma_1(x,y) \leq 2 \cdot T_r\big(\frac{1}{1-\delta}\big).
\end{equation}
\end{fact}
\begin{proof}
From \cref{EQ:decomp}, we have:
\[
    (1-x^2)\sigma_1(x,y) \leq (1-x^2)(1-y^2) + \epsilon \quad \text{ for } x, y \in [-1, 1].
\]
As $\delta \ge \epsilon$, we thus get:
\[
    \sigma_1(x,y) \leq (1-y^2) + \frac{\epsilon}{1-x^2} \leq 1 + 1 = 2 \quad \text{ for } x^2 \leq 1 - \delta,~y^2 \leq 1 - \delta.
\]
In other words, we have $\max_{\|(x,y)\|_\infty^2 \leq 1-\delta} |\sigma_1(x,y)| \leq 2$. We may therefore apply \Cref{LEM:Markovscaled} to $\sigma_1$ to obtain the fact.
\end{proof}

\begin{fact} \label{FACT:secondderivative}
Let $g(t) = \sigma_1(1 + t, 1 - t)$. Then for any $1 > \delta \geq \epsilon$, and any $u \in [-\delta, \delta]$, we have:
\[
    \frac{1}{2}|g''(u)| \leq r^4 \cdot T_r\big(\frac{1}{1-\delta}\big)^2.
\]
\end{fact}
\begin{proof}
Assume w.l.o.g. that $u \geq 0$. Note that $\frac{1}{1-\delta} \geq 1 + \delta \geq 1 + u$. Using \cref{EQ:markovk>0}, \cref{EQ:sigma1small}, and~\Cref{LEM:Markov}, we therefore have that:
\begin{align*}
    |g''(u)| &= \dv[2]{t} \big(\sigma_1(1+u+t, ~1-u-t)\big)\bigr|_{t=0}
    \\
    &\leq T_r^{(2)}(1+u) \cdot \max_{x,y \in [-1, 1]} \sigma_1(x, y)
    \\
    &\leq r^4 \cdot T_r\big(\frac{1}{1-\delta}\big) \cdot 2T_r\big(\frac{1}{1-\delta}\big) \qedhere. 
\end{align*}
\end{proof}

\begin{fact} \label{FACT:derivsmall}
Let $g(t) = \sigma_1(1 + t, 1 - t)$. Then for any $1 > \delta \geq \epsilon$, we have:
\[
g'(0) \leq \frac{\epsilon}{2\delta} r^2 + \delta r^4 \cdot T_r\big(\frac{1}{1-\delta}\big)^2.
\]
\end{fact}
\begin{proof}
Assume $g'(0) \geq 0$ (otherwise the statement is trivial). Note that $g(t) \geq 0$ for all $t \in \R$. By Taylor's theorem, there exists $u \in [-\delta, 0]$ such that:
\begin{align*}
    &0 \leq g(-\delta)  = g(0) - g'(0) \cdot \delta + \frac{1}{2} g''(u) \cdot \delta^2, \\
    &\implies g'(0) \leq \frac{g(0)}{\delta} + \frac{1}{2} |g''(u)| \cdot \delta \leq \frac{\epsilon r^2}{2\delta} + \frac{1}{2} |g''(u)| \cdot \delta,
\end{align*}
where we have used that $g(0) = \sigma_1(1, 1) \leq \frac{1}{2}\epsilon r^2$ by \Cref{FACT:sigma1small}. Now apply \Cref{FACT:secondderivative} to conclude the proof.
\end{proof}
We are ready to conclude the argument. Let $g(t) = \sigma_1(1+t, 1-t)$. By Taylor's theorem, there exists a ${u \in [0, \delta]}$ such that:
\begin{align}
    g(\delta) &= g(0) + g'(0) \cdot \delta + \frac{1}{2} g''(u) \cdot \delta^2 \notag \\
    &\leq \frac{1}{2}\epsilon r^2 + \bigg( \frac{1}{2} \epsilon r^2 + \delta^2 r^4 \cdot T_r\big(\frac{1}{1-\delta}\big)^2 \bigg) + \delta^2 r^4 \cdot T_r\big(\frac{1}{1-\delta}\big)^2 \notag \\
    \label{EQ:pfnegativ1}
    &= \epsilon r^2 + 2\delta^2 r^4 \cdot T_r\big(\frac{1}{1-\delta}\big)^2,
\end{align}
where we have used \Cref{FACT:sigma1small}, \Cref{FACT:secondderivative} and \Cref{FACT:derivsmall} to get the inequality.
Now set $\delta = \sqrt{\epsilon} \geq \epsilon$. In light of~\cref{EQ:decomp}, and since $\delta \leq 1$, we have
\begin{align*}
    -3 \delta \cdot g(\delta) &\leq (1-(1+\delta)^2) \cdot g(\delta) \leq f(1 + \delta, 1 - \delta) + \epsilon \leq -4\delta^2 + \delta^4 + \epsilon \leq -2\epsilon,
    \\
    &\implies g(\sqrt{\epsilon}) = g(\delta) \geq \frac{2}{3} \sqrt{\epsilon}.
\end{align*}
Using~\cref{EQ:pfnegativ1}, we thus find that:
\begin{align} \label{EQ:negativfinal}
\frac{2}{3} \sqrt{\epsilon} \leq \epsilon r^2 + 2\epsilon r^4 \cdot T_r\big(\frac{1}{1-\sqrt{\epsilon}}\big)^2, \quad
\implies \frac{1}{3\sqrt{\epsilon}} \leq \frac{1}{2}r^2 + r^4 \cdot T_r\big(\frac{1}{1-\sqrt{\epsilon}}\big)^2.
\end{align}
We may assume that $r = O(1/\sqrt[4]{\epsilon})$ (otherwise there is nothing to prove), in which case \Cref{LEM:chebybound} tells us that $T_r\big(\frac{1}{1-\sqrt{\epsilon}}\big)^2 = O(1)$. But then \cref{EQ:negativfinal} implies that:
\[
    r = \Omega(1/\sqrt[8]{\epsilon}). \qedhere
\]
\end{proof}

\begin{proof}[Proof of \Cref{THM:negative}]
    It remains to see that \Cref{PROP:negative} implies \Cref{THM:negative}, which is rather straightforward. Indeed, any decomposition of $(1-x_1^2)(1-x_2^2) + \epsilon$ in $\quadr(\cube{n})_r$, $n \geq 3$, immediately gives a decomposition of $(1-x_1^2)(1-x_2^2) + \epsilon$ in $\quadr(\cube{2})_r$ by setting $x_3 = \ldots = x_n = 0$ (see also the proof of \Cref{cor::lower}).
\end{proof}

\section{Discussion}
\label{SEC:Discussion}
We have proven an upper bound on the required degree of a Putinar-type representation of a positive polynomial on $\cube{n}=[-1, 1]^n$, described using the inequalities $1-x_1^2, \dots , 1-x_n^2$, of the order $O(\funcmax/\funcmin)$, see \Cref{THM:main}. 
This result improves upon the previously best known bound of $O((\funcmax/\funcmin)^{10})$, obtained from the general result \Cref{COR:GeneralPutinar}.
Complementing this upper bound, we have exhibited a family of polynomials $f=f_\epsilon$ of degree~$4$ with $\funcmax = 1+\epsilon$, $\funcmin = \epsilon$ whose Putinar-type representations are necessarily of degree at least $\Omega( \sqrt[8]{\funcmax / \funcmin}) = \Omega(1/\sqrt[8]{\epsilon})$, see \Cref{THM:negative}.
These results have direct application in polynomial optimization, see
\Cref{cor::upper} and \Cref{cor::lower}.

We remark that the same asymptotic results hold true if we describe $\cube{n}$
using the inequalities $1 \pm x_i$ for $i=1, \dots n$ instead of $1 - x_i^2$. This follows from the identities:
\begin{align*}
    1 \pm x_i & = \frac{1}{2} \left( (1 \pm x_i)^2 + 1 - x_i^2 \right) \\
    1 - x_i^2 & =\frac{1}{2}\left( (1-x_i)^2(1+x_i) + (1+x_i)^2(1-x_i)\right)
\end{align*}

Hereafter we describe more connections of these results with existing literature and 
propose some possible future research directions.

\medskip\noindent\textbf{Improving the upper degree bound.}
In the proof of \Cref{THM:main}, we use an effective Schm\"udgen's Positivstellensatz on a scaled hypercube $[-\eta, \eta]^n$; namely \Cref{cor::scled_schmudgen}. This corollary is responsible for the term of order $O(\sqrt{\funcmax/\funcmin})$ in our result.
\Cref{cor::scled_schmudgen} could be replaced with any other effective Schm\"udgen's Positivstellensatz on $[-\eta, \eta]^n$ with sufficiently good rate of convergence, and this could lead to improvements of the final result. In particular, the dependence on $n, d$ of the constant $C(n,d)$ appearing in \Cref{cor::scled_schmudgen}  is quite bad (see \cite[Eq.~(18)]{LaurentSlot:hypercube}), especially compared to the constant $d^2(\log n)$ we introduce in our proof of \Cref{THM:main}. Combing the proof of \Cref{THM:main} with a better effective Schm\"udgen's Positivstellensatz on $[-\eta, \eta]^n$ would lead to an effective Putinar's Positivstellensatz on $\cube{n}$ that is asymptotically interesting also for $n, d \to \infty$.

\medskip\noindent\textbf{Logarithmic degree bounds.}
In their recent work~\cite{Bach:hypercube}, Bach \& Rudi give an alternative proof of \Cref{THM:boxSchmudgen}, working from the perspective of \emph{trigonometric polynomials}. A trigonometric polynomial (of degree $2d$) is a 1-periodic function of the form:
\[
    F(\x) = \sum_{\omega \in \Z^n} \widehat F(\omega) \cdot \exp(2 i \pi \omega^\top \x),
\]
with $\hat F(\omega) = 0$ whenever $\|\omega\|_\infty > d$. Minimization of a polynomial $f \in \R[\x]_{2d}$ over the unit hypercube $\cube{n}$ is equivalent to minimizing the trigonometric polynomial
\[
    F(\x) = f(\cos 2 \pi x_1, \cos 2 \pi x_2, \ldots, \cos 2 \pi x_n)
\]
over $[0, 1]^n$. Moreover, it turns out that $f$ has a Schm\"udgen-type certificate of nonnegativity with sum-of-squares multipliers of degree $2r$ when $F$ can be written as a sum of squares of trigonometric polynomials of degree $r$, see \cite{Bach:hypercube}.

Bach \& Rudi show bounds with \emph{logarithmic} dependence in $F_{\max} / F_{\min}$ on the required degree~$r$ in a (trigonometric) sum-of-squares representation for a class of positive trigonometric polynomials on~$[0, 1]^n$ satisfying a strict local optimality condition. This result translates immediately to the setting of Schm\"udgen-type representations for (regular) positive polynomials on $[-1, 1]^n$ with the appropriate local optimality condition. 
It would be interesting to see if such assumptions might lead to better degree bounds for the \emph{Putinar}-type representations as well. We remark that the polynomials $\prod_{i=1}^n {(1-x_i^2)} + \epsilon$ featured in \Cref{PROP:negative} ($n=2$) and in \Cref{CONJ:dKL} ($n \geq 2$) do not satisfy this condition. In fact, \Cref{PROP:negative} shows that it is \emph{not} possible to achieve logarithmic degree bounds for representations of general polynomials in the quadratic module (it is not known whether this is possible for the preordering).

Another interesting question is what the results of this work and \cite{LaurentSlot:hypercube}, \cite{Bach:hypercube} suggest to be the best choice of certificate to use for polynomial optimization in practice. On the one hand, Putinar-type certificates lead to smaller SDPs. On the other hand, the error guarantees for Schm\"udgen-type and trigonometric certificates are much better (especially when assuming local optimality conditions). The situation is summarized in \Cref{TAB:comparison}. Adequately addressing this question would require a numerical study, which is an interesting direction for future research.

{
\rowcolors{2}{white}{gray!25}
\begin{table}[]
\renewcommand{\arraystretch}{1.2} 
\begin{tabular}{ccccc}
                \textbf{certificate}& \textbf{\# matrices} & \begin{tabular}[c]{@{}c@{}}\textbf{matrix size} \end{tabular}  & \textbf{error} & \begin{tabular}[c]{@{}c@{}}\textbf{error} (local opt.)\tablefootnote{Here, $a > 0$ and $b > 1$ are constants depending on $n, d$, and local properties of $f$ around its global minimizer. See~\cite[{Thm 4.1}]{Bach:hypercube} for details.}\end{tabular}\\
Putinar         & $n+1$ & ${n + r \choose r}$ & $O(1/r)$    & $O(1/r)$ \\
Schm\"udgen       & $2^n$ & ${n + r \choose r}$ & $O(1/r^2)$  & $O\big(\exp(-(a \, r)^{b}) \big)$ \\
Trigonometric   & $1$   & $(2r + 1)^n$        & $O(1/r^2)$  & $O\big(\exp(-(a \, r)^{b}) \big)$\\
\end{tabular}
\caption{Overview of computational complexity and best-known asymptotic error guarantees for approximation hierarchies for polynomial optimization on $[-1, 1]$ based on different certificates. Shown are the number and size of the matrices occurring in the canonical SDP formulation of each of the hierarchies, when using sum-of-squares multipliers of degree~$2r$. Note that for fixed $n \in \N$ and $r \to \infty$, we have ${n + r \choose r} \approx r^n$ and $(2r + 1)^n \approx 2^n \cdot r^n$. } 
\label{TAB:comparison}
\end{table}
}

\medskip\noindent\textbf{Stability and lower degree bounds for the Positivstellens\"atze.}
{As we have seen above, quantitatively comparing degree bounds for different representations of nonnegative polynomials  provides limitations to the convergence rate of optimization schemes. But despite this usefulness, quantitative \emph{lower} degree bounds have not been extensively investigated: to stimulate research in this direction, hereafter we summarize the known results, with an emphasis on the connection to the concept of \emph{stability}.}

Recall that a quadratic module $\quadr(\vb g)$ is \emph{Archimedean} if there exists an $R \in \R$ such that $R - \metric{\x}{2} \in \quadr(\vb g)$. Clearly, $\quadr(\cube{n})$ is Archimedean,
since $n - x_1^2 - \ldots - x_n^2 \in \quadr(\cube{n})$. Putinar's Positivstellensatz tells us that 
 if $\quadr (\vb g)$ is Archimedean and $f>0$ on $\cS(\vb g)$, then $f \in \quadr(\vb g)$.
As we have shown in \Cref{THM:negative} and \Cref{SEC:negative} the degree needed for the representation $f \in \quadr(\cube{n})$ may go to infinity as $\funcmax / \funcmin$ goes to infinity for $n \ge 2$, even if the degree of $f$ is fixed. 

A strictly related concept is \emph{stability}, introduced in \cite{powers_moment}.
We say that the quadratic module $\quadr(\vb g)$ is \emph{stable} if for all $d \in \N$ there exists a $d \le k \in \N$ such that $\quadr(\vb g) \cap \R[\x]_{\le d} \subseteq \quadr(\vb g)_k$.
\Cref{THM:negative} (through \Cref{PROP:negative}) shows that $\quadr(\cube{n})$ is non-stable for $n \ge 2$:
indeed, the degree needed for the representation of $(1-x_1^2)(1-x_2^2)+\epsilon \in \quadr (\cube{n})$
depends on $\epsilon$ and not only on the degree $d=4$ and $n$.
We can regard \Cref{THM:negative}, \Cref{PROP:negative} and the result of Stengle~\cite{Stengle:negative} as quantitative versions of the \emph{non-}stability property.
We now give an overview of known results relating Archimedean and stability properties.

We start with the one dimensional case, i.e. quadradic modules and preorderings that are subsets of $\R[x]$
(for the more general case of quadratic modules and preorderings defining semialgebraic sets on real curves,
see \cite{scheiderer_sums_2003, scheiderer_non-existence_2005, scheiderer_semidefinite_2018}). Recall that in $\R[x]$ every finitely generated quadratic module defining a compact semialgebraic set is an Archimedean preordering, see \cite{scheiderer_distinguished_2005}.
The result of Stengle \cite{Stengle:negative} shows that there are compact, one dimensional subsets of the real line which are defined by a (finitely generated) preordering that is non-stable. This is also an example of an Archimedean quadratic module that is non-stable. 
This happens because the choice for the generator of the preordering is not the \emph{natural} one,  see \cite{kuhlmann_positivity_2002, marshall_positive_2008}. The generator also does not satisfy the constraint qualification conditions.
Indeed, if the preordering defining the compact set contains the natural generators, then the
preordering is stable. This follows from a direct computation as in \cite[Prop.~2.7.3]{marshall_positive_2008} or applying \cite[Cor.~3.18]{scheiderer_non-existence_2005}.
The converse is not true in general: the preordering $\preo (-x^2)$ is stable (and Archimedean) but it does not contain the natural generators $\pm x$ of the origin.
See \cite[Thm.~9.3.3]{marshall_positive_2008} for a generalization of the idea of natural generators.

We turn our attention to the two dimensional case. Every Archimedean preordering defining a semialgebraic subset of $\R^2$ with nonempty interior is non-stable, see \cite[Thm.~5.4]{scheiderer_non-existence_2005} and also \cite[Ex.~5.1]{scheiderer_non-existence_2005}.
Notice that in \cite[Ex.~5.1]{scheiderer_non-existence_2005},
a family of strictly positive polynomials 
and an interior point of the semialgebraic set is used to prove non-stability, while in \Cref{PROP:negative} we use a boundary point.
In particular, the results in \cite{scheiderer_non-existence_2005} apply to both $\quadr(\cube{2})$ and $\preo (\cube{2})$, which are therefore non-stable. We recall also that, despite being non-stable, $\preo (\cube{2})$ is \emph{saturated}, i.e. $\preo (\cube{2}) = \nnon (\cube{2})$ (see \cite{scheiderer_sums_2006} or \cite[Thm.~9.4.5]{marshall_positive_2008}). On the contrary, $\quadr (\cube{2}) \subsetneq \nnon (\cube{2})$. This is an important difference and it is exploited in 
\Cref{PROP:negative} to prove the lower bound for the representation in $\quadr (\cube{n})$.
We do not know if a quantitative version of \cite[Ex.~5.1]{scheiderer_non-existence_2005}, that applies also to the preordering $\preo(\cube{2})$, would give better or worse bounds compared to the bound of \Cref{PROP:negative}. In general, quantitatively comparing \Cref{PROP:negative} and \cite[Ex.~5.1]{scheiderer_non-existence_2005} could be the first step to understand if the lower degree bounds for representations in $\preo(\vb g)$ and $\quadr (\vb g)$ are significantly different.
{In particular, this investigation could help answer the following question: can we find a family of polynomials showing that an exponential convergence for the moment-SOS Schm\"udgen-type hierarchy on $\cube{2}$ (or, more generally, on $\cube{n}$) is not always possible \emph{without} local optimality conditions? (see \Cref{TAB:comparison}).}

For $\vb g$ defining a compact semialgebraic set $\cS(\vb g)$ of dimension $\ge 3$, the preordering $\preo (\vb g)$ is non-stable \cite{scheiderer_non-existence_2005} and it is not saturated, i.e. $\preo (\vb g) \subsetneq \nnon (\cS(\vb g))$. The same results hold true for Archimedean quadratic modules $\quadr (\vb g)$.

{Let us finally recall that the stability and non-stability properties for degree one polynomials are strongly related to the exactness and convergence of Lasserre's spectrahedral approximation of semialgebraic sets, see \cite{netzerExposedFacesSemidefinitely2010, gouveiaPositivePolynomialsProjections2011}. There, other examples of non-stable quadratic modules are studied using boundary points of semialgebraic sets, but no quantitative estimates are given.}


\medskip\noindent\textbf{Acknowledgments.}
The authors wish to thank Claus Scheiderer for the detailed discussion about Archimedean and stability properties for preorderings and quadratic modules. We further wish to thank Monique Laurent for her helpful comments and suggestions on an early version of this work, and Bernard Mourrain for useful discussions. We thank Etienne de Klerk for pointing out a flaw in the statement of \Cref{LEM:Markov}.
The authors finally wish to thank the anonymous referees for their insightful remarks, which helped to improve a preliminary version of this manuscript. 

The first author was partially funded by the European Union’s Horizon 2020 research and innovation programme POEMA, under the Marie Skłodowska-Curie Actions, grant agreement 813211, and by the Paris \^{I}le-de-France Region, under the grant agreement 2021-02--C21/1131.

\bibliographystyle{myamsplain}
\bibliography{bibLucas}
 \appendix
 \section{Explicit expressions for degree shifts}
 \label{app::degree_shift}

 In this appendix we discuss \cref{eq::direct}:
 \[
     1 - x^2 = \frac{(q-1) - qx^2 + x^{2q}}{q} + \frac{1-x^{2q}}{q}
 \]
 that is a key ingredient for the proof of the upper bound. Hereafter we provide some related explicit formulae and
 their consequences in the proof of \Cref{thm::from_preordering_to_module}.

 Despite its simplicity, it is difficult to derive this kind of expressions.
 Indeed, this is a representation for $1 - x^2 \in \quadr (1-x^{2q})$
 (see \cref{eq::sos} below for an explicit sum-of-squares expression of $(q-1) - qx^2 + x^{2q}$).
 Obtaining \emph{exact} representations for polynomials in quadratic modules is challenging,
 even in the univariate case, and to the authors' best knowledge there is currently no software available
 to solve the problem in general.

 We therefore discuss in more detail how \cref{eq::direct} was obtained, and provide explicit 
 expressions for \Cref{lem::basic} and \Cref{lem::basic2}.

 Consider the equation:
 \begin{equation*}
     1 - x = \frac{1}{2}\left( (1 - x)^2 + 1 - x^2 \right) \in \quadr(1-x^{2})_{2}
 \end{equation*}
 If we substitute $x = x^2$ we obtain 
 \begin{equation*}
     1 - x^2 = \frac{1}{2}\left( (1 - x^2)^2 + 1 - x^4 \right) \in \quadr(1-x^{4})_{4}
 \end{equation*}
 More generally, substituting $x = x^{2^{m-1}}$, we have:
 \begin{equation*}
     1 - x^{2^{m-1}} = \frac{1}{2}\left( (1 - x^{2^{m-1}})^2 + 1 - x^{2^{m}} \right) \in \quadr(1-x^{2^m})_{2^m}
 \end{equation*}
 It is then possible to obtain the explicit formula:
 \begin{equation}\label{eq::recurrence}
     1 - x^2 = \sum_{i=1}^{m-1} \left(\frac{1}{2^i} (1 - x^{2^i})^2 \right) + \frac{1}{2^{m-1}}(1 - x^{2^m}) \in \quadr(1 - x^{2^m})_{2^m}
 \end{equation}
 which is equivalent to \cref{eq::direct} with $2q = 2^m$.

 Using \cref{eq::recurrence}, we can deduce also the explicit expression for $\eta^2 - x_k^2 \in \quadr(n - \metric{\x}{2^m})_{2^m}$ in
 \Cref{lem::basic2} :
 \begin{equation}\label{eq::deformation}
     \eta^2 - x_k^2 = \eta^2 \sum_{i=1}^{m-1} \left(\frac{1}{2^i} \left(1 - \left(\frac{x_k}{\eta^2}\right)^{2^i}\right)^2 \right) + \frac{\eta^2}{2^{m-1} n}\left(\sum_{j\neq k}x_j^{2^m} + n - \metric{\x}{2^m}\right)
 \end{equation}

 We have therefore seen that for $2q = 2^m$ the necessary representations can be derived easily.
 It is then possible to make an educated guess to avoid the power of $2$, writing $1-x^2 = f_q + \frac{1-x^{2q}}{q}$. We then obtain the  polynomial $f_{q} = \frac{(q-1) - q x^2 + x^{2q}}{q}$, that is nonnegative and thus a sums of squares. An explicit way to note this is by writing:
 \begin{equation*}
     \begin{cases}
         f_1 = 0 \\
         f_{q+1} = \frac{q}{q+1}x^2 f_q + \frac{q}{q+1} (1-x^2)^2
     \end{cases}
 \end{equation*}
 or more directly:
 \begin{equation}
         \label{eq::sos}
     f_q = \sum_{i=1}^{q-1} \frac{q-i}{q}x^{2(i-1)}(1-x^2)^2
 \end{equation}
 Therefore the final explicit expression for $1 - x^2 \in \quadr (1-x^{2q})_{2q}$ in \Cref{lem::basic} is:
 \begin{equation*}
     1-x^2 = \left( \sum_{i=1}^{q-1} \frac{q-i}{q}x^{2(i-1)}(1-x^2)^2 \right) + \frac{1-x^{2q}}{q} \in \quadr (1-x^{2q})_{2q}
 \end{equation*}
 We can deduce also an explicit expression for \Cref{lem::basic2}, i.e. for 
 $\eta^2 - x_i^2 \in {\quadr(n - \metric{\x}{2q})_{2q}}$:
 \begin{equation}\label{eq::explicit_preo_module}
     \eta^2 - x_i^2 = \eta^2 f_q\left(\frac{x_i}{\eta}\right) + \frac{\eta^2}{q n} \left( \sum_{j\neq i} x_j^{2q} + n - \metric{\x}{2q} \right)
 \end{equation}
 with $f_q$ as in \cref{eq::sos}. The equations \cref{eq::explicit_preo_module} and \cref{eq::sos} give also explicit expressions for the inclusion
 \[
     \preo(\eta^2 - x_1^2, \dots , \eta^2 - x_n^2)_{k} \subseteq\quadr(1 - x_1^2, \dots , 1 - x_n^2)_{k + n (2q-2)}
 \]
 in \Cref{thm::from_preordering_to_module}.

\end{document}